\documentclass[10pt, final]{amsart}

\usepackage{amsmath}
\usepackage{amsthm} %breqn loads amsmath
\usepackage{graphicx}
\usepackage{color}
\usepackage{scalerel}
\usepackage[dvipsnames]{xcolor}
\usepackage{tikz-cd}
\usepackage[T1]{fontenc}

\usepackage{lmodern}

\usepackage[backend=biber,
url=false,
isbn=false,
backref=true,
citestyle=alphabetic,
bibstyle=alphabetic,
autocite=inline,
maxnames=99,
minalphanames=4,
maxalphanames=4,
sorting=nyt,]{biblatex}
\usepackage{enumerate}

\usepackage{amsfonts}
\usepackage{amssymb}
\usepackage[hidelinks]{hyperref}
\usepackage[capitalize]{cleveref}

\usepackage[activate={true,nocompatibility},final,tracking=true,kerning=true,spacing=true,stretch=10,shrink=10]{microtype}
\microtypecontext{spacing=nonfrench}

\usepackage[normalem]{ulem}

\usepackage{mathtools}

\DeclareMathOperator*{\forkindep}{\raise0.2ex\hbox{\ooalign{\hidewidth$\vert$\hidewidth\cr\raise-0.9ex\hbox{$\smile$}}}}

\newcommand{\tp}{\operatorname{tp}}

\DeclareTextCommand{\DZ}{OT2}{D2}
\newcommand{\dfg}{\mathrm{dfg}}
\newcommand{\fsg}{\mathrm{fsg}}
\newcommand{\id}{\mathrm{id}}

\newcommand{\stab}{\mathrm{stab}}
\newcommand{\img}{\mathrm{img}}

\newcommand{\Lc}{\mathcal{L}}
\newcommand{\cU}{\mathcal{U}}

\newtheorem*{claim-star}{Claim}

\newenvironment{clmproof}[1][\proofname]{\proof[#1]}{\endproof}

\newtheorem*{theorem-non}{Theorem}
\newtheorem{theorem}{Theorem}[section] % numbered like the section
\newtheorem{lemma}[theorem]{Lemma}

\newtheorem{prop-def}[theorem]{Proposition-Definition}
\newtheorem{corollary}[theorem]{Corollary}
\newtheorem{fact}[theorem]{Fact}
\newtheorem{fact-eh}[theorem]{Fact(?)}

\newtheorem{question}[theorem]{Question}
\newtheorem{proposition}[theorem]{Proposition}
\newtheorem{proposition-eh}[theorem]{Proposition(?)}
\newtheorem*{theorem-star}{Theorem}
\newtheorem*{conjecture-star}{Conjecture}
\newtheorem*{lemma-star}{Lemma}

\theoremstyle{definition}
\newtheorem{definition}[theorem]{Definition}
\newtheorem{example}[theorem]{Example}

\newtheorem{remark}[theorem]{Remark}

\theoremstyle{remark}

\newcommand{\fs}{\mathrm{fs}}

\newcommand{\inv}{\mathrm{inv}}

\newcommand{\cM}{\mathcal{M}}

\allowdisplaybreaks %Fix horrible vertical spacing issues.

\title{%
    \rotatebox{180}{\reflectbox{Upside down and backwards}}%
}

\author[K. Gannon]{Kyle Gannon}
\thanks{Partially supported by the Fundamental Research Funds for the Central Universities, Peking University, grant no. 7100604835 and by the National Natural Science Fund of China, grant no. 12501001}
\address{Beijing International Center for Mathematical Research (BICMR) \\ Peking University \\ Beijing, China.}
\urladdr{\href{https://orcid.org/0000-0001-5951-5269}{\includegraphics[height=\fontcharht\font`\B]{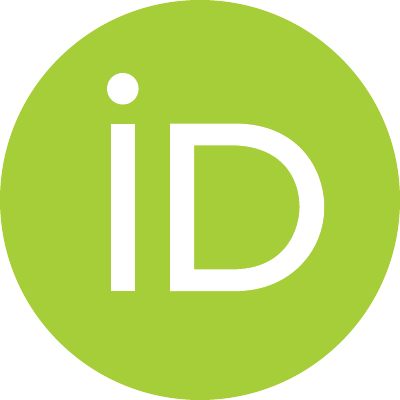} {0000-0001-5951-5269}}\\
\href{http://faculty.bicmr.pku.edu.cn/~kyle/}{http://faculty.bicmr.pku.edu.cn/~kyle/}}
\email{kgannon@bicmr.pku.edu.cn}

\author[T. Rzepecki]{Tomasz Rzepecki}
\address{Instytut Matematyczny \\ Uniwersytet Wrocławski \\ Wrocław, Poland }
\urladdr{\href{https://orcid.org/0000-0001-9786-1648}{\includegraphics[height=\fontcharht\font`\B]{orcidlogo.pdf} {0000-0001-9786-1648}}\\
\href{https://fricas.org/~rzepecki/}{https://fricas.org/~rzepecki/}}
\email{Tomasz.Rzepecki@math.uni.wroc.pl}

\begin{document}

\begin{abstract} We investigate the semigroup of invariant types through the lens of Ellis theory; primarily focusing on definably amenable NIP groups. In this context, we observe that the collection of strong right $f$-generic types forms the unique minimal left ideal and thus, the Ellis subgroups are isomorphic to $G/G^{00}$ via the canonical quotient map. As consequence of the Newelski-Pillay conjecture, the Ellis subgroups of the semigroup of invariant types are \emph{abstractly isomorphic} to the Ellis subgroups of the semigroup of finitely satisfiable types in the definable amenable NIP setting. We are interested in the existence of \emph{natural isomorphisms} from invariant Ellis subgroups to finitely satisfiable Ellis subgroups and we determine when these isomorphisms can be witnessed by variants of the canonical NIP retraction map. Several limiting examples are provided. Outside of the NIP context, we provide an abelian group (and thus definably amenable) with an $\emptyset$-definable (dfg) type
in which the invariant Ellis subgroups and finitely satisfiable Ellis subgroups not isomorphic.
\end{abstract}

\maketitle

\section{Introduction}

A large part of the model theory research involving topological dynamics focuses on the space of global types which are finitely satisfiable in a fixed small model. This is a seemingly appropriate setting because this space of types, equipped with the \emph{Newelski product}, is canonically isomorphic to the Ellis semigroup of a particular group action. However, the Newelski product naturally extends to the space of global types which are invariant over a fixed small model. Similarly, this space of types with the extended product forms a left-continuous compact Hausdorff semigroup and is thus also susceptible to \emph{Ellis theory analysis}. Here, we take this extended product seriously and study the minimal ideals and Ellis subgroups of the semigroup of invariant types as well as their connections to the semigroup of finitely satisfiable types. As stated in the abstract, we primarily focus on definably amenable NIP groups.

In the context of definably amenable NIP groups, the Ellis theory of the semigroup of invariant types is quite well-behaved and easily understood. We observe the following, strengthening and recontextualizing the observations of \cite[Section 2]{pillay2013topological}:

\begin{theorem}\label{theorem:one} Suppose that $T$ is NIP. Let $\cU$ be a monster model of $T$ and $M$ a small elementary substructure of $\cU$. Let $G(x)$ be an $\emptyset$-definable group which is definably amenable. We let $\mathcal{F}_{r}$ denote the collection of types which are strong right $f$-generic over $M$. Then
\begin{enumerate}
    \item $\mathcal{F}_{r}$ is the unique minimal left ideal of $S_{G}^{\inv}(\cU,M)$. As consequence of general theory, $\mathcal{F}_{r}$ is a two-sided ideal and for every strong right $f$-generic $t$, we have that $t * S^{\inv}_{G}(\cU,M)$ is an Ellis subgroup of $S^{\inv}_{G}(\cU,M)$.
    \item The Ellis subgroups of $S^{\inv}_{G}(\cU,M)$ are all isomorphic to $G(\cU)/G^{00}(\cU)$ via the map induced from the canonical quotient. As consequence of the Newelski-Pillay conjecture for definably amenable NIP groups (proved by Chernikov and Simon \cite{CS}), the Ellis subgroups of $(S_{G}^{\inv}(\cU,M),*)$ and $(S_{G}^{\fs}(\cU,M),*)$ are abstractly isomorphic.
    \end{enumerate}
    See Theorem~\ref{theorem:main} for details.
\end{theorem}

In the above result, one notices that the \emph{strong right $f$-generics} play a central role in the analysis. This is in contrast to the majority of the literature where the \emph{strong left $f$-generics} are central. We remark that results involving this kind of \emph{left-right phenomenon} between finitely satisfiable and invariant objects have been implicitly observed in the context of Keisler measures (e.g., see \cite[Theorem 5.1(9)]{chernikov2023definable} or Proposition \ref{prop:implicit}).

By Theorem \ref{theorem:one}, the invariant Ellis subgroups and the finitely satisfiable Ellis subgroups are abstractly isomorphic in the definably amenable NIP setting. What interests us is whether there exists \emph{natural} isomorphisms between some/all of the Ellis subgroups of these semigroups. In the NIP setting, there exists a continuous retraction map (introduced by Simon \cite{simon2015invariant}) from the collection of global $M$-invariant types to the collection of global finitely satisfiable types in $M$ and thus the natural question arises:

\begin{question} How does the retraction map interact with the Newelski product? When does the retraction and/or variants of this map witness an isomorphism between Ellis subgroups?
\end{question}

Surprisingly, the answer is a little complicated. In the NIP $\fsg$ case, the unique minimal left ideal of $S_{G}^{\inv}(\cU,M)$ coincides with the unique minimal left ideal of $S_{G}^{\fs}(\cU,M)$ and thus the retraction map is clearly an isomorphism of Ellis subgroups (i.e., the retraction is equal to the identity on the minimal left ideal). Additionally, we show that when the underlying group is abelian, then one can find Ellis subgroups in which the retraction map is an isomorphism. We then provide an NIP $\dfg$ group where the retraction map does not map invariant Ellis subgroups to finitely satisfiable Ellis subgroups. However, an interesting phenomenon occurs in the NIP $\dfg$ setting: the invariant Ellis subgroups appear \emph{upside down} and \emph{backwards} with respect to the finitely satisfiable Ellis subgroups. More concretely, there exists invariant Ellis subgroups such that the \emph{model theoretic inverse} of the retraction map is an anti-isomorphism while, at the same time, the retraction map may fail to map invariant Ellis subgroups to finitely satisfiable Ellis subgroups. We recall that an anti-isomorphism between two groups $G$ and $H$ is a bijection $f$ such that for every $g,g' \in G$, $f(g \cdot g') = f(g') \cdot f(g)$. Anti-isomorphisms are in one-to-one correspondence with isomorphisms via precomposition with group inversion. We prove the following:

\begin{theorem}\label{theorem:intro} Suppose that $T$ is NIP. Let $\cU$ be a monster model of $T$ and $M$ be a small elementary substructure of $\cU$. Let $G(x)$ be a $\emptyset$-definable group which is definably amenable. We let $F_{M}: S_{G}^{\inv}(\cU,M) \to S_{G}^{\fs}(\cU,M)$ denote the retraction map from invariant types to finitely satisfiable types. We let $K_{M} = F_{M}(-)^{-1}$ where $p^{-1} = \tp(a^{-1}/\cU)$ when $a \models p$. Then
\begin{enumerate}
\item If $G$ is $\fsg$, then $F_{M}$ is (trivially) an isomorphism. In particular, the unique minimal left ideal in the space of invariant types coincides with the unique minimal left ideal in the space of finitely satisfiable types and the retraction map restricted to this ideal is just the identity. (See Proposition~\ref{prop:fsgmain})
\item If $G$ is abelian and $G(M)$ contains coset representatives for each coset of $G^{00}(\cU)$ in $G(\cU)$, then for both the retraction map and the inverted retraction map, there exists pairs of Ellis subgroups (one invariant, one finitely satisfiable) such that the map in question is an isomorphism. (See Theorem~\ref{thm:invert_iso_abelian}.)
\item If $G$ is $\dfg$ and $t$ is a $\dfg$ type over $M$, then $K_{M}|_{t * S^{\inv}_{G}(\cU,M)}$ is a (continuous) anti-isomorphism of Ellis subgroups. If $G$ is also abelian, then $F_{M}|_{t*S_{G}^{\inv}(\cU,M)}$ is an isomorphism of Ellis subgroups. Moreover, $ t * S_{G}^{\inv}(\cU,M)$ and $K_{M}(t * S_{G}^{\inv}(\cU,M))$ are compact groups under the induced topology and $K_{M}|_{t*S_{G}^{\inv}(\cU,M)} \circ \inv$ is an isomorphism of topological groups (where $\inv$ is the group theoretic inverse). In the abelian case, $F_{M}|_{t * S_{G}^{\inv}(\cU,M)}$ is also an isomorphism of compact topological groups. (See Theorem~\ref{theorem:map}, Corollary~\ref{cor:isocompact} and Proposition~\ref{prop:abelian_dfg}. Cf.\ also Theorem~\ref{thm:main_dfg_alternative} for a partial generalization.)
\end{enumerate}
\end{theorem}

As limiting examples, we provide an NIP $\dfg$ group such that the retraction map does not map any invariant Ellis subgroups to any finitely satisfiable Ellis subgroups as well as an NIP $\fsg$ group such that the inverted retraction map does not map any invariant Ellis subgroups to any finitely satisfiable Ellis subgroups.

We softly conjecture that the analysis from Theorem \ref{theorem:intro} may be extended to all definably amenable groups in NIP theories. Toward this conjecture, we compute invariant Ellis subgroups and construct explicit isomorphisms for semidirect products of $\fsg$ and $\dfg$ groups. We recall that Pillay and Yao asked whether every definably amenable group (definable in a distal theory) is a definable extension of an $\fsg$ group by a $\dfg$ group \cite[Question 1.19]{pillay2016minimal}. Thus this work in the semidirect product setting is a substantial step toward resolving this broader question. We first give a construction of invariant Ellis semigroups as internal products of the $\dfg$ and $\fsg$ invariant Ellis semigroups. We then build a natural isomorphism from invariant Ellis semigroups to finitely satisfiable ones. Intuitively, the analysis in this particular semidirect product setting arises from Theorem \ref{theorem:intro} -- on the $\fsg$ portion of the group, \emph{do nothing}; on the $\dfg$ portion of the group, \emph{retract and invert}.

Outside of the NIP setting, we provide an example of a group such that the invariant Ellis subgroups and finitely satisfiable Ellis subgroups are not isomorphic. The group in question is the atomless Boolean algebra with symmetric difference as multiplication. This group is clearly abelian and has the independence property (it is also definably amenable, even dfg).

The paper is outlined as follows. In Section 2, we provide some basic background and necessary preliminaries on Ellis theory and model theoretic dynamics. In Section 3, we prove that if $G(x)$ is an definably amenable group in an NIP theory, then $S_{G}^{\inv}(\cU,M)$ has a unique minimal left ideal, namely the collection of strong right $f$-generics (Proposition \ref{prop:unique_min}). By general Ellis theory, this implies that the Ellis subgroups are of the form $p *S_{G}^{\inv}(\cU,M)$ where $p$ is any strong right $f$-generic. Using the fact that the right stabilizer of a strong right $f$-generic is $G^{00}(\cU)$, we then prove that the Ellis subgroups of $S_{G}^{\inv}(\cU,M)$ are isomorphic to $G(\cU)/G^{00}(\cU)$ (Theorem \ref{theorem:main}). In Section 4, we focus on the retraction map and its inverted retraction. We first prove some basic properties connecting the retraction map with the Newelski product. We then give some general criterion for whether the retraction or inverted retraction forms an isomorphism/anti-isomorphism (Lemmas \ref{lemma:codomain}, \ref{lemma:isoequiv} and \ref{lemma:antiequiv}). A straightforward application of these lemmas show that if $G(x)$ is an abelian NIP group and $G(M)$ is rich enough, then both the retraction and the inverted retraction witness isomorphisms from invariant Ellis subgroups to finitely satisfiable Ellis subgroups (Theorem \ref{thm:invert_iso_abelian}). We then analyze the NIP $\fsg$ and $\dfg$ cases. In the NIP $\fsg$ case, the unique minimal left ideals coincide and the retraction map is (trivially) an isomorphism (Proposition \ref{prop:fsgmain}). This particular result follows quite quickly from previously analysis on $\fsg$ groups (e.g., see \cite{hrushovski2012note,pillay2013topological}) and is more or less folklore. In the NIP $\dfg$ case, we prove that if $p \in S_{G}^{\inv}(\cU,M)$ is a $\dfg$ type then the map $K_M|_{p * S_{G}^{\inv}(\cU,M)}$ is an anti-isomorphism of Ellis subgroups (Theorem \ref{theorem:map}). We then argue that if we precompose this map with group inversion, then the resulting map is a topological/algebraic isomorphism between compact groups. More generally, a variant of the statement holds in the NIP setting when the Ellis subgroups are closed (Theorem \ref{thm:main_dfg_alternative}). We then provide the following limiting examples:
\begin{enumerate}
\item An NIP $\dfg$ group such that the retraction map does not send invariant Ellis subgroups to finitely satisfiable Ellis subgroups $(\mathbb{R} \rtimes \mathbb{Z}/2\mathbb{Z})$. Hence the retraction map does not always witness an isomorphism of Ellis subgroups (Example \ref{example:dfg}).
\item An NIP $\fsg$ group such that the inverse retraction map does not send invariant Ellis subgroups to finitely satisfiable Ellis subgroups $(S^{1} \rtimes \mathbb{Z}/2\mathbb{Z})$. Hence the inverse retraction map does not always witness an anti-isomorphism of Ellis subgroups (Example \ref{example:fsg}).
\item An NIP definably amenable group such that neither the retraction nor the inverse retraction map invariant Ellis subgroups to finitely satisfiable Ellis subgroups $((\mathbb{R} \times S^{1}) \rtimes \mathbb{Z}/2\mathbb{Z})$; Example \ref{example:product}).
\item A definable amenable group (abelian, $\dfg$) group such that the invariant Ellis subgroups are not isomorphic to the finitely satisfiable Ellis subgroups (atomless Boolean algebra) (Corollary \ref{cor:different}).
\end{enumerate}

In the final section, we analysis semidirect product of $\fsg$ groups and $\dfg$ groups. We compute invariant Ellis subgroups in the semidirect product in terms of the $\dfg$ and $\fsg$ invariant Ellis subgroups (Corollary \ref{corollary:constuct}). We also give an explicit isomorphism between invariant Ellis subgroups and finitely satisfiable Ellis subgroups in which we \emph{retract and invert} the $\dfg$ portion of the group and \emph{do nothing} on $\fsg$ portion (Theorem \ref{thm:main_semidirect}).

\section*{Acknowledgement} We would like to thank Artem Chernikov, Daniel Hoffmann, Grzegorz Jagiella, Krzysztof Krupi\'{n}ski, Pierre Simon, Ningyuan Yao, and Zhantao Zhang for helpful discussions. We also thank Itay Kaplan for essentially asking \emph{What is known about $(S_{G}^{\inv}(\cU,M),*)$?} during a talk at the Fields institute in Toronto.

\section{Preliminaries}

Our notation is standard and we refer the reader to \cite{Guide} for background on NIP theories. We let $T$ be a first-order theory in a language $\Lc $. We let $\cU$ be a monster model of $T$ and $M$ be a small elementary submodel of $\cU$. We often let $x,y,z$ denote tuples of variables. We let $G(x)$ denote a $\emptyset$-definable group with respect to the theory $T$. If $A \subseteq \cU$ then we let $S_{x}(A)$ denote the collection of complete types (over $A$ in variable(s) $x$).

\subsection{Ellis theory}

We are interested in objects of the form $(X,*)$ where $X$ is a compact Hausdorff topological space and $*$ is a semigroup operation on $X$ which is a left-continuous, i.e., for every $a \in X$, the map $-*a:X \to X$ is continuous. Unimaginatively, we call $(X,*)$ a \emph{left-continuous compact Hausdorff semigroup}. A left (right) ideal $I$ of $X$ is a non-empty subset of $X$ such that $XI \subseteq I$ (reps. $IX \subseteq I$). A minimal left (right) ideal is a left (right) ideal which does not properly contain another left (resp. right) ideal. The next fact is, by now, \emph{well-known} in the model theory community. It describes the structure of minimal left ideals in left-continuous compact Hausdorff semigroups (e.g., see \cite[A.8]{rzepecki2018bounded} or \cite{Glasner:Proximal_flows}).

\begin{fact}\label{fact:Ellis} Suppose that $(X,*)$ is a left-continuous compact Hausdorff semigroup. Then there exists a minimal left ideal $I$. Moreover,
\begin{enumerate}
    \item Every element of $I$ generates $I$ as a left ideal, i.e., for any $p \in I$, $X * p = I$. As consequence, $I$ is closed.
    \item There exists an element $u \in I$ such that $u$ is an idempotent, i.e., $u*u = u$. Moreover, we say that an idempotent $u$ is \textbf{minimal} if there exists some minimal left ideal $J \subseteq X$ such that $u \in J$.
    \item If $u \in I$ is an idempotent, then $u * I$ is a group. We call groups of this form Ellis subgroups of $I$. An \textbf{Ellis subgroup} is an Ellis subgroup of $J$ for some minimal left ideal $J$ of $X$.
    \item $I$ can be written as a disjoint union of Ellis subgroups of $I$, i.e., $I = \bigsqcup_{u \in \id(I)} u *I$ where $\id(I)$ is the collection of idempotents in $I$.
    \item Given two idempotents $u,v \in I$ the map $u*-|_{v*I} \to u*I$ is an (algebraic) isomorphism of groups. Hence every Ellis subgroup of $I$ is isomorphic as abstract groups.
    \item If $J$ is another minimal left ideal, then any Ellis subgroup of $I$ is isomorphic to any Ellis subgroup of $J$. Hence the isomorphism type of the Ellis subgroups are independent of both choice of minimal left ideal and of idempotent. The isomorphism type of any particular Ellis subgroup is called the ideal group.
    \item If $J$ is another minimal left ideal, then for any idempotent $v \in J$ there exists a unique idempotent $u \in I$ such that $u*v = v$ and $v*u = u$.
\end{enumerate}
\end{fact}

What is less commonly known is the structure of the \emph{minimal right ideals} in such kinds of semigroups. The following statements are \emph{folklore-adjacent}.

\begin{fact}\label{fact:right} Let $(X,*)$ be a left-continuous compact Hausdorff semigroup.
\begin{enumerate}
    \item Let $L$ be a left ideal. Every right ideal $R \subseteq X$ contains a right ideal of the form $u*X$ where $u$ is an idempotent in $L$.
    \item Every minimal right ideal is of the form $v* X$ where $v$ is a minimal idempotent.
    \item If $u$ is a minimal idempotent then $u * X$ is a minimal right ideal.
    \item Every right ideal contains a minimal right ideal.
    \item If $u$ is a minimal idempotent and $I$ is a minimal left ideal, then $u * X \cap I$ is an Ellis subgroup of $I$.
    \item Suppose that $u$ is an idempotent and $R = u * X$ is a minimal right ideal. Then $u$ is a minimal idempotent from some minimal left ideal $I$ and $R * u$ is an Ellis subgroup of $I$.
    \item
    Every left [right/two-sided] ideal $\cM$ contains an ideal of the form $X*x$ [$x*X$/$X*x*X$] with $x\in \cM$, in particular, a left [right/two-sided] ideal is minimal if and only if it has no such proper left [right/two-sided] subideals.
    \item
    If $\cM$ is a minimal left [right] ideal and $x\in X$ is arbitrary, then $\cM*x$ [$x*\cM$] is a minimal left [right] ideal.
    \item
    The union of all the minimal left ideals, and the union of all the minimal right ideals are each equal to the minimal (two-sided) ideal.
    \item
    The intersection of any minimal left ideal $L$ and any minimal right ideal $R$ is an Ellis subgroup and it is equal to the product $R*L$.
    \item
    Every minimal right ideal is a disjoint union of Ellis subgroups.
\end{enumerate}
In general, minimal right ideals need not be closed sets.
\end{fact}

\begin{proof} We prove the statements:
\begin{enumerate}
    \item First note that every right ideal contains a principal right ideal, i.e., one of the form $p * X$ for some $p \in X$. Indeed, simply choose $p$ to be any element of $R$. Hence, it suffices to show that any principal right ideal contains one of the form described in the statement. So consider $s *X$ where $s \in X$. Let $L' \subseteq L$ be a minimal left ideal and $q \in L'$. Notice that $s *X \supseteq (s * q) *X$. However, $s * q \in L'$ and then it is an element of some Ellis subgroup, say $v * L'$ where $v \in L'$ is an idempotent. Hence, we claim that $(s * q) * X = v * X$. We conclude that $v * X \subseteq s *X$, completing the proof.
    \item Follows directly from the proof of Statement (1).
    \item Let $L = X * u$. Consider the right ideal $u * X$. Let $R \subseteq u *X$ be a right subideal. By Statement (1), there exists some $v \in L$ such that $v * X \subseteq R$. Thus, $v *X \subseteq R \subseteq u *X$. Since $u,v$ are idempotents in $L$, $u * v = u$ and $v * u = v$ which implies that $v *X = u *X$ and so $R = u *X $.
    \item Directly from (1) and (3).
    \item By (7) of Fact \ref{fact:Ellis}, there exist some $v \in I$ such that $u * v = v$ and $v * u = u$. It follows that $u * X = v * X$ and so $u * X \cap I = v *X \cap X * v= v * X * v$, which is an Ellis subgroup of $I$.
    \item Direct from the observation that $R * u = u * X * u \subseteq I$.
    \item
    Immediate: take any $x\in \cM$.
    \item
    Suppose $\cM$ is a minimal left ideal (the right version is symmetric). Then $\cM*x$ is clearly a left ideal, and for any $y_1,y_2\in \cM$ we have that $x*y_2\in \cM$, so there is (by minimality of $\cM$) some $y_3\in X$ such that $y_3*x*y_2=y_1$, so that $y_3*x*y_2*x=y_1*x$. As $y_1,y_2\in \cM$ are arbitrary, this shows that the ideal $X*y_2*x$ contains $\cM*x$, and so $\cM*x$ is a minimal left ideal.
    \item
    Let $I$ be the union of all minimal left ideals (which exist by Fact~\ref{fact:Ellis}).
    Note that if $\cM$ is any minimal left ideal, then $\cM*X = I$: indeed, containment holds by the preceding statement and reverse containment holds since if $\cM'$ is another minimal left ideal, then $\cM*\cM'=\cM'$. It follows that $I$ is a two-sided ideal, and if $x\in I$ is arbitrary, then $X*x$ is a minimal left ideal, so $X*x*X=I$, which shows that $I$ is a minimal two-sided ideal. The proof of the analogous statement for right ideals is analogous (using (4) for the existence of minimal right ideals). The equality follows from the observation that in any semigroup, if $I_1,I_2$ are (two-sided) ideals, then $I_1*I_2\subseteq I_1\cap I_2$ and $I_1\cap I_2$ is an ideal, so there is at most one minimal ideal.
    \item
    Let $L$ be a minimal left ideal and $R$ be a minimal right ideal. Suppose $u\in L\cap R$ is idempotent. Then (by (7)) $L=X*u$ and $R=u*X$, so $X*u\cap u*X=L\cap R$, so $u$ is an identity element of $L\cap R$. On the other hand, $L\cap R\supseteq R*L$, so it is nonempty, and it is a right ideal of $(L,*)$, so (by 2.1(4)) it is a union of Ellis subgroups. On the other hand, since every idempotent in $L\cap R$ is an identity element of $L\cap R$, there cannot be more than one, so $L\cap R$ is in fact an Ellis subgroup. To see the equality with the product, just notice that $R\cap L=u*L\subseteq R*L\subseteq R\cap L$.
    \item
    Let $R$ be a minimal right ideal, and let $I$ be the minimal (two-sided) ideal (as in (9)). Then
    \begin{equation*}R=R\cap I=R\cap \left(\bigcup_{L \in \mathbb{L}} L \right)=\bigcup_{L \in \mathbb{L}} R\cap L,
    \end{equation*} where $\mathbb{L}$ is the collection of all minimal left ideals. Each $R\cap L$ is an Ellis subgroup by above. \qedhere
\end{enumerate}
\end{proof}

\subsection{Model theoretic dynamics} Here we give some necessary background on invariant types, the Newelski product, and definable amenability. Fix $\cU$ a monster model of a first order theory $T$ and let $M$ be a small elementary submodel. We say that $p$ in $S_{x}(\cU)$ is invariant over $M$ if for every $\Lc $-formula $\varphi(x,y)$ and parameters $a, b \in \cU^{y}$ such that $a \equiv_{M} b$, we have that $\varphi(x,a) \in p$ if and only if $\varphi(x,b) \in p$. We say that $p$ in $S_{x}(\cU)$ is finitely satisfiable in $M$ if whenever $\varphi(x) \in p$, there exists some $a \in M^{x}$ such that $\cU \models \varphi(a)$. It is straightforward to check that if a type is finitely satisfiable in $M$ then it is invariant over $M$. Let $G(x)$ be a $\emptyset$-definable subgroup. We let $S_{G}(\cU) = \{p \in S_{x}(\cU): p \vdash G(x)\}$. Consider the following two spaces of types.

\begin{enumerate}
\item $S_{G}^{\inv}(\cU,M) := \{p \in S_{G}(\cU): p$ is $M$-invariant$\}$.
\item $S_{G}^{\fs}(\cU,M) := \{p \in S_{G}(\cU): p$ is finitely satisfiable in $M\}$.
\end{enumerate}

We remark that $S_{G}^{\fs}(\cU,M) \subseteq S_{G}^{\inv}(\cU,M)$ and that both spaces are compact Hausdorff spaces with the topology induced from the standard stone space topology. We now recall the Newelski product (originally introduced into model theory in \cite{N1}) which turns these spaces into left-continuous compact Hausdorff semigroups.

\begin{definition} Suppose that $p\in S_{G}^{\inv}(\cU,M), q\in S_G(\cU)$. The \textbf{Newelski product} of $p$ with $q$, denoted $p * q$, is defined as follows: For any $\Lc $-formula $\varphi(x,z)$ and parameters $d$ from $\cU^{z}$, we have that $\varphi(x,d) \in p * q$ if and only if $\varphi(x \cdot y,d) \in p_{x} \otimes q_{y}$ if and only if $\varphi(x \cdot b,d) \in p$ where $b \models q|_{Md}$. We recall that `$\otimes$' is the Morley product of invariant types (see \cite[Chapter 2]{Guide} for basic properties).
\end{definition}

It is straightforward to verify the following properties of the Newelski product and so we leave it as an exercise for the reader.

\begin{fact}\label{fact:nb} The following statements are true.
\begin{enumerate}
    \item If $p,q \in S_{G}^{\inv}(\cU,M)$, then $p * q \in S_{G}^{\inv}(\cU,M)$.
       \item If $p,q \in S_{G}^{\fs}(\cU,M)$, then $p * q \in S_{G}^{\fs}(\cU,M)$.
    \item If $p,q \in S_{G}^{\inv}(\cU,M)$ and both $p$ and $q$ are $M$-definable, then $p * q$ is $M$-definable.
    \item For any $q \in S_{G}(\cU)$, the map $-*q :S_{G}^{\inv}(\cU,M) \to S_{G}(\cU)$ is continuous.
    \item If $p \in S_{G}^{\inv}(\cU,M)$ is $M$-definable, then $p*-:S_{G}(\cU) \to S_{G}(\cU)$ is continuous.
    \item If $p,q\in S^\inv_G(\cU,M)$ and $r\in S_G(\cU)$, then $(p*q)*r=p*(q*r)$.
\end{enumerate}
We conclude that both $(S_{G}^{\inv}(\cU,M),*)$ and $(S_{G}^{\fs}(\cU,M),*)$ are left-continuous compact Hausdorff semigroups.
\end{fact}

\begin{definition} Suppose that $p \in S_{G}(\cU)$. Then a global left translate of $p$ is a type of the form $c \cdot p := \{\varphi(c^{-1} \cdot x): \varphi(x) \in p\}$ where $c \in G(\cU)$. Alternatively, for any $\varphi(x) \in \Lc _{x}(\cU)$ such that $\varphi(x) \vdash G(x)$,
\begin{equation*}
    \varphi(x) \in c \cdot p \Longleftrightarrow \varphi(c \cdot x) \in p.
\end{equation*}
Likewise, a global right translate of $p$ is a type of the form $p \cdot c = \{\varphi(x \cdot c^{-1}): \varphi(x) \in p\}$ and so for any $\varphi(x) \in \Lc _{x}(\cU)$ and $\varphi(x) \vdash G(x)$,
\begin{equation*}
    \varphi(x) \in p \cdot c \Longleftrightarrow \varphi(x \cdot c) \in p.
\end{equation*}
\end{definition}

We now recall several different kinds of global invariant types.

\begin{definition} Suppose that $p \in S_{G}^{\inv}(\cU,M)$. We say that
\begin{enumerate}
    \item $p$ is \textbf{strong left [right] $f$-generic over $M$} if every global left [right] translate of $p$ is $M$-invariant.
    \item $p$ is \textbf{left [right] $\dfg$ over $M$} if $p$ is $M$-definable and every global left [right] translate of $p$ is $M$-definable.
    \item $p$ is \textbf{left [right] $\fsg$ over $M$} if $p$ is finitely satisfiable in $M$ and every global left [right] translate of $p$ is finitely satisfiable in $M$.
\end{enumerate}
\end{definition}

Obviously, both left [right] $\fsg$ and $\dfg$ (over $M$) types are left [resp. right] strong $f$-generic (over $M$). In the NIP setting, a type $p \in S_{G}^{\inv}(\cU,M)$ is left $\fsg$ over $M$ if and only if $p$ is right $\fsg$ over $M$ (\cite[Proposition 8.29]{Guide} or \cite[Proposition 4.2]{NIP1}). This left-right collapse is not true for strong left/right $f$-generics nor left/right $\dfg$ types, i.e., see Example \ref{example:Heisenberg}. The following characterization of \emph{definable amenability} can be found in \cite[Corollary 8.20]{Guide}.

\begin{fact}\label{fact:def_amen} Suppose that $T$ is NIP. Then the following are equivalent:
\begin{enumerate}
\item $G$ is definably amenable, i.e., there exists/for any model $N$ of $T$, there exists a $G(N)$-left invariant Keisler measure $\mu$ such that $\mu(G(x)) = 1$.
\item There exists $p$ in $S_{G}(\cU,M)$ such that $p$ is strong left $f$-generic over $M$.
\end{enumerate}
\end{fact}

$f$-generic types can also be characterized in different ways. One useful characterization to keep in mind is the following (see e.g.\cite[Theorem 1.2]{CS}).
\begin{fact}\label{fct:fgen_small_orbit}
    A type $p\in S^\inv_G(\cU,M)$ is left [right] $f$-generic of and only if $G(\cU)\cdot p$ [$p\cdot G(\cU)$] is small.
\end{fact}

\begin{definition}\label{def:f} Suppose that $T$ is NIP and $G(x)$ is definably amenable. We let $\mathcal{F}_{r}$ denote the collection of strong right $f$-generics types over $M$.
\end{definition}

\begin{definition}\label{def:mti} Suppose that $p \in S_{G}^{\inv}(\cU,M)$. Then we defined the \emph{model theoretic inverse} of the type $p$ as $p^{-1} := \tp(a^{-1}/\cU)$ where $a \models p$.
\end{definition}

The following fact is left to the reader.

\begin{fact}\label{fact:easy} Let $p \in S_{G}^{\inv}(\cU,M)$. Then
\begin{enumerate}
\item $p$ is strong left $f$-generic if and only if $p^{-1}$ is strong right $f$-generic.
\item $p$ is strong right $f$-generic if and only if $p^{-1}$ is strong left $f$-generic.
\end{enumerate}
\end{fact}

For any small subsets $A$ of $\cU$, we let $G_{A}^{00}(\cU)$ be the smallest $A$-type-definable subgroup of $G(\cU)$ of bounded index.

\begin{definition} If $T$ is NIP, then for every small $A \subseteq \cU$, $G_{A}^{00}(\cU) = G_{\emptyset}^{00}(\cU)$ (see \cite{shelah2008minimal}). We will write ${G}_{\emptyset}^{00}(\cU)$ simply as ${G}^{00}(\cU)$. We let $\pi:G(\cU) \to G(\cU)/G^{00}(\cU)$ be the quotient map and $\hat{\pi}: S_{G}(\cU) \to G(\cU)/G^{00}(\cU)$ be the canonical extension. We recall that $\hat{\pi}|_{S^{\inv}_{G}(\cU,M)}: S^{\inv}_{G}(\cU,M) \to G(\cU)/G^{00}(\cU)$ is a continuous semigroup homomorphism.
\end{definition}

A proof of the next fact can be found in \cite[Lemma 8.18]{Guide} (see \cite[Proposition 5.6]{NIP2} for primary source).

\begin{fact}\label{fact:g00}
Suppose that $p \in S_{G}^{\inv}(\cU,M)$ and $p$ is a strong left $f$-generic type. Then the left stabilizer of $p$ denoted $\stab_{l}(p)$ is precisely $G^{00}(\cU)$. In other words, if we let,
\begin{equation*}
    \stab_{l}(p) := \{ g \in G(\cU): g \cdot p = p\},
\end{equation*}
then $\stab_{l}(p) = G^{00}(\cU)$.
\end{fact}

\begin{corollary}\label{fact:stabiliser_right_gen_g00}
Suppose that $p \in S_{G}^{\inv}(\cU,M)$ and $p$ is strong right $f$-generic type. Then the right stabilizer of $p$ denoted $\stab_{r}(p)$ is precisely $G^{00}(\cU)$. In other words, if we let,
\begin{equation*}
    \stab_{r}(p) := \{ g \in G(\cU): p \cdot g = p \},
\end{equation*}
then $\stab_{r}(p) = G^{00}(\cU)$.
\end{corollary}

\begin{proof} Follows directly from Facts \ref{fact:easy} and \ref{fact:g00}. Indeed, fix $p$ a strong right $f$-generic type. For any formula $\varphi(x) \in \Lc _{x}(\cU)$ such that $\varphi(x) \vdash G(x)$ and $g \in G^{00}(\cU)$, we have that
\begin{align*}
\varphi(x) \in p \cdot g \Longleftrightarrow \varphi(x \cdot g) \in p &\Longleftrightarrow \varphi( (g^{-1} \cdot x^{-1})^{-1}) \in p \Longleftrightarrow \varphi(g^{-1} \cdot x^{-1}) \in p^{-1}\\ &\Longleftrightarrow \varphi(x^{-1}) \in g^{-1} \cdot p^{-1} \Longleftrightarrow \varphi(x^{-1}) \in p^{-1} \Longleftrightarrow \varphi(x) \in p.
\end{align*}
Thus $G^{00}(\cU) \subseteq \stab_{r}(p)$.

On the other hand, if $\varphi(x)$ is as above, $g \in \stab_{r}(p)$, and $\psi(x) := \varphi(x^{-1})$, then
\begin{align*}
\varphi(x) \in g \cdot p^{-1} \Longleftrightarrow \varphi(g \cdot x) \in p^{-1} &\Longleftrightarrow \varphi(g \cdot x^{-1}) \in p \Longleftrightarrow \psi(x \cdot g^{-1}) \in p \\
&\Longleftrightarrow \psi(x) \in p \cdot g^{-1} \Longleftrightarrow \psi(x) \in p \Longleftrightarrow \varphi(x) \in p^{-1},
\end{align*}
and so $\stab_{r}(p) \subseteq \stab_{l}(p^{-1}) = G^{00}(\cU)$.
\end{proof}

Our final fact is the Newelski-Pillay conjecture which was resolved by Chernikov and Simon (see \cite[Theorem 5.7]{CS}).
\begin{fact}\label{fact:NC} Suppose that $T$ is NIP and $G(x)$ is definably amenable. Let $I$ be a minimal left ideal of $S_{G}^{\fs}(\cU,M)$. Then for any idempotent $u \in I$ we have that the map $\hat{\pi}|_{u * I}: u * I \to G(\cU)/G^{00}(\cU)$ is a group isomorphism.
\end{fact}

\section{Ellis analysis of definably amenable groups in NIP theories}

In this section, we fix an NIP theory $T$ and let $G(x)$ be a $\emptyset$-definable group which is definably amenable. We let $\cU$ be a monster model of $T$ and $M$ be a small elementary submodel of $\cU$. We first prove that that the collection of strong right $f$-generic types (over $M$) is the unique minimal left ideal in $S_{G}^{\inv}(\cU,M)$. Using this observation and the fact that the right stabilizer of any strong right $f$-generic is $G^{00}(\cU)$, we prove that the Ellis subgroups of $S^{\inv}_{G}(\cU,M)$ are all isomorphic to $G(\cU)/G^{00}(\cU)$ using the canonical extension of the standard quotient map, i.e., $\hat{\pi}$. To be frank, the hardest part of the proof is realizing that the strong right $f$-generics are the object of interest.

\begin{proposition}\label{prop:fgen_ideal}
    The set of strong right $f$-generics (over $M$) is a left ideal in $S_{G}^{\inv}(\cU,M)$.
\end{proposition}

\begin{proof} By Fact \ref{fact:def_amen} and Fact \ref{fact:easy}, $S_{G}^{\inv}(\cU,M)$ contains a strong right $f$-generic type. If $p$ is a strong right $f$-generic, $r \in S_{G}^{\inv}(\cU,M)$, and $g \in G(\cU)$, then
\begin{equation*} (r * p) \cdot g = r * (p \cdot g),
\end{equation*}
by associativity of the Newelski product. Since $p$ is right $f$-generic, $p \cdot g$ is $M$-invariant, and so $r * ( p \cdot g)$ is also $M$-invariant. So any global right translate of $(r * p)$ is $M$-invariant and thus $(r * p)$ is a strong right generic.
\end{proof}

The previous proposition does not requite NIP, only the existence of strong $f$-generics. Our next proposition does require NIP.

\begin{proposition}[NIP]\label{prop:unique_min} The set of right $f$-generics (over $M$) is the unique minimal left ideal in $S_{G}^{\inv}(\cU,M)$.
\end{proposition}
\begin{proof}
By Proposition \ref{prop:fgen_ideal}, the collection of strong right $f$-generic types forms a left ideal. By Fact \ref{fact:Ellis}, the semigroup $(S_{G}^{\inv}(\cU,M),*)$ contains a minimal left ideal. It suffices to show that the ideal of strong right $f$-generics is contained in every minimal left ideal. Suppose that $L_0$ is a minimal left ideal. Then $L_0$ contains an idempotent, say $u$. Recall that by Fact~\ref{fact:stabiliser_right_gen_g00}, the right-stabilizer of any strong right $f$-generic type is $G^{00}(\cU)$. Thus, for any strong right $f$-generic type $q$, $\Lc (\cU)$-formula $\varphi(x,c)$, and $b \models u|_{Mc}$, we have that $b \in G^{00}(\cU)$ [since $\hat{\pi}:S_{G}^{\inv}(\cU,M) \to G(\cU)/G^{00}(\cU)$ is a homomorphism of semigroups and so $\hat{\pi}(u)$ must be the identity] and so,
\begin{equation*}
    \varphi(x,c) \in (q * u) \Longleftrightarrow \varphi(x\cdot b,c) \in q \Longleftrightarrow \varphi(x,c) \in q \cdot b \Longleftrightarrow \varphi(x,c) \in q.
\end{equation*}
Thus $q \in L_0$. Hence $L_0$ contains all strong right $f$-generics and the statement holds.
\end{proof}

The following fact is proved for the reader's convenience.

\begin{fact}\label{fact:equiv} Let $(X,*)$ be a left-continuous compact Hausdorff semigroup. Then the following are equivalent:
\begin{enumerate}
\item There exists a unique minimal left ideal, $I$.
\item Minimal right ideals are precisely Ellis subgroups (and vice versa).
\item There exists a minimal left ideal $I$ which is also a right ideal.
\end{enumerate}
\end{fact}

\begin{proof} $(1) \to (2)$. Fix a minimal right ideal $R$. By Fact \ref{fact:right}(7), $R$ is the union of Ellis subgroups. Thus $R$ is a subset of $I$. By Fact \ref{fact:right}(1), we have that $R = v *X$ for some $v \in I$. We claim that for any other idempotent $u \in I$ such that $u \neq v$, $u \not \in v *X$. Notice that if $u \in v *X$, then $u = v * q$ for some $q \in X$. Thus $u = u *v = v *q * v \in v * X * v$, but this is impossible since $v$ is the identity of the Ellis subgroup $v *X *v$. Thus, $R$ contains at most one idempotent and so $R$ must be an Ellis subgroup of $I$. It follows that every Ellis subgroup is a minimal right ideal by Fact \ref{fact:right}(3) and the analysis above.

$(2) \to (3)$. Let $I$ be any minimal left ideal. Then $I$ is the union of right ideals and thus a right ideal.

$(3) \to (1)$. Suppose that there exists another minimal left ideal, $L$. Since $I$ and $L$ are both minimal left ideals, it is easy to check that $I \cap L = \emptyset$. However, since $I$ is a right ideal, we have that $\emptyset \neq I * L \subseteq I \cap L$, a contradiction.
\end{proof}

\begin{lemma}\label{lemma:apply} Suppose that $p \in S_{G}^{\inv}(\cU,M)$ is a strong right $f$-generic. Then for any $q_1,q_2 \in S_{G}(\cU)$, if $\hat{\pi}(q_1) = \hat{\pi}(q_2)$, then $p * q_1 = p * q _2$. In particular, if $g \in G(\cU)$ and $\pi(g) = \hat{\pi}(q_1)$ then $p * q_1 = p \cdot g$.
\end{lemma}

\begin{proof}
    Fix an $\Lc (\cU)$-formula $\varphi(x,b)$ such that $\varphi(x,b) \vdash G(x)$. For $i=1,2$, let $g_i\models q_i|_{Mb}$. Then
    \[g_1G^{00}(\cU)=\hat\pi(q_1)=\hat\pi(q_2)=g_2G^{00}(\cU),\] hence $p\cdot g_1=p\cdot g_2$.
    Using this observation and the definition of $*$ we have
    \[
        \varphi(x,b)\in p*q_1 \Longleftrightarrow \varphi(x,b)\in p\cdot g_1=p\cdot g_2 \Longleftrightarrow \varphi(x,b)\in p*q_2.
    \]
    For the \emph{in particular} part, we remark that $p \cdot g = p * \tp(g/\cU)$.
\end{proof}

We note that parts of the following theorem, namely the fact that in the right $f$-generics form a left-simple compact left topological semigroup, the right orbits are its Ellis subgroups and that they are isomorphic to $G(\cU)/G^{00}(\cU)$, were observed in \cite[Section 2]{pillay2013topological}.
\begin{theorem}\label{theorem:main} Suppose that $p \in S_{G}^{\inv}(\cU,M)$ is a strong right $f$-generic. Then
\begin{enumerate}
    \item $p * S_{G}^{\inv}(\cU,M)$ is an Ellis subgroup of $S_{G}^{\inv}(\cU,M)$.
    \item $p * S_{G}^{\inv}(\cU,M) = p \cdot G(\cU)=p*S_G(\cU)$.
    \item Let $B \subseteq G(\cU)$ be a set of (unique) representatives of $G^{00}(\cU)$-cosets in $G(\cU)$. Then $p \cdot G(\cU) = p \cdot B$.
    \item The map $\hat{\pi} : p \cdot B \to G(\cU)/G^{00}(\cU)$ is a group isomorphism.
\end{enumerate}
As consequence, the map $\hat{\pi}|_{p * S_{G}^{\inv}(\cU,M)}$ is a group isomorphism.
\end{theorem}

\begin{proof} Without loss of generality, suppose that $p$ is an idempotent.
\begin{enumerate}
    \item Since $S_{G}^{\inv}(\cU,M)$ has a unique minimal left ideal, i.e., the strong right $f$-generics over $M$ (Proposition \ref{prop:unique_min}), the minimal right ideals are Ellis subgroups of this ideal. By Fact \ref{fact:right}(3), $p * S_{G}^{\inv}(\cU,M)$ is a minimal right ideal and by Fact \ref{fact:equiv}, $p * S_{G}^{\inv}(\cU,M)$ is an Ellis subgroup.
    \item Note that $\hat \pi$ and its restriction to $S_G^\inv(\cU,M)$ are onto $G(\cU)/G^{00}(\cU)$, then apply Lemma~\ref{lemma:apply}.
    \item By Fact \ref{fact:stabiliser_right_gen_g00}, $\stab_{r}(p) = G^{00}(\cU)$. Let $c \in G(\cU)$. We want to show that $p \cdot c \in p \cdot B$. Consider the unique element $a \in B$ such that $\pi(a) = \pi(c)$. We claim that $p \cdot a = p \cdot c$. Indeed, since they are in the same coset, there exists some element $\xi$ of $G^{00}(\cU)$ such that $\xi \cdot c = a$. Then $\theta(x) \in p \cdot a$ iff $\theta(x \cdot a) \in p$ iff $\theta(x \cdot \xi \cdot c) \in p$ iff $\theta(x \cdot c) \in p$ iff $\theta(x) \in p \cdot c$. Hence $p \cdot G(\cU) \subseteq p \cdot B$. The other direction is obvious.
    \item The map $\hat{\pi}$ is a homomorphism from $S_{G}^{\inv}(\cU,M)$ to $G(\cU)/G^{00}(\cU)$. It suffices to prove that this map restricted to $p \cdot B$ is a bijection.
    \begin{enumerate}
        \item Injective: Let $a_1,a_2 \in B$ and suppose that $\hat{\pi}(p \cdot a_1) = \hat{\pi}(p \cdot a_2)$. Then $\pi(a_1) = e \cdot \pi(a_1) = \hat{\pi}(p) \cdot \pi(a_1) = \hat{\pi}(p \cdot a_1) = \hat{\pi}(p \cdot a_2) = \ldots = \pi(a_2)$. Then $a_1$ and $a_2$ are in the same coset, but this is impossible since $B$ is a collection of unique reps.
        \item Surjective: $\hat{\pi}$ maps the appropriate coset representative to the appropriate coset. \qedhere
    \end{enumerate}
\end{enumerate}
\end{proof}

After having proved this result, the second author and Daniel Hoffmann proved a similar theorem in the context of semigroups of invariant types coming from the actions of automorphism groups -- see \cite[Theorem 4.10]{HR25}.

\subsection{Left, right and bi \texorpdfstring{$f$}{f}-generics} Here we take the opportunity to make some observation involving left, right, and bi $f$-generics. We do not assume NIP in this subsection.
\begin{proposition}\label{prop:left_f_generic}
    Suppose $G(x)$ is a definable group. Then the set of strong left $f$-generics in $S^\inv_G(\cU,M)$, if it is nonempty, is a right ideal.
\end{proposition}
\begin{proof}
    Take $p,q\in S^\inv_G(\cU,M)$ with $p$ a strong left $f$-generic, and any $g\in G(\cU)$. Then $g\cdot (p*q)=(g\cdot p)*q$, hence $p*q$ is a strong left $f$-generic.

    Indeed, fix any $N\preceq \cU$ containing $M$ and $g$, $g_2\models q|N$ and $g_1\models p|Ng_2$. Then $gg_1\models (g\cdot p)|Ng_2$ (so $gg_1g_2\models ((g\cdot p)*q)|N$) and $gg_1g_2\models (g\cdot (p*q))|N$.
\end{proof}

\begin{corollary}
    If there are left (equivalently, right) $f$-generic types in $S^\inv_G(\cU,M)$, then there is an Ellis subgroup consisting of types which are both left and right $f$-generic.
\end{corollary}
\begin{proof}
    By Proposition \ref{prop:left_f_generic}, the left $f$-generics in $S^\inv_G(\cU,M)$ form a right ideal in $S^\inv_G(\cU,M)$, whereas by Proposition \ref{prop:fgen_ideal}, the right $f$-generics in $S^\inv_G(\cU,M)$ form a left ideal.
    Using Fact \ref{fact:right}(4), we conclude that there is a minimal right ideal $R$ consisting of left $f$-generics, and a minimal left ideal $L$ consisting of strong right $f$-generics, and $R\cap L$ is an Ellis subgroup (Fact \ref{fact:right}(10)). (Under NIP, we have by Proposition \ref{prop:unique_min} that in fact $L$ is the set of all strongly right $f$-generic types and $R\subseteq L$, so the Ellis subgroup is $R=R\cap L$.)
\end{proof}

\begin{example}\label{example:Heisenberg}
    In general, not every strong left $f$-generic is a strong right $f$-generic (and vice versa).

    Let $M$ be a real closed field and let $G$ be the Heisenberg group, i.e.\ $G$ consists of matrices of the form
    \[
        [a,b,c]\coloneqq \begin{pmatrix}1&a&b\\0&1&c\\0&0&1\end{pmatrix},
    \]
    with $a,b,c$ in the field.

    Let $t_{+\infty}$ be the $\emptyset$-definable type of an infinite element in the field. For $(x,y,z)\models t_{+\infty}^{\otimes 3}$ (so that $x\gg y\gg z\gg \cU$), let $p=\tp([x,y,z]/\cU)\in S_G(\cU)$. Clearly, $p$ is $\emptyset$-definable (in particular, it is in $S_G^\inv(\cU,M)$).

    Now, given any $[a,b,c]\in G(\cU)$, direct computation shows that
    \[
        [a,b,c]\cdot [x,y,z]=[a+x,b+y+az,c+z]\equiv_\cU [x,y,z],
    \]
    so $p$ is a left $G(\cU)$-invariant, and hence a strongly left $f$-generic (in fact, dfg) type. Since the left stabilizer of $p$ is $G(\cU)$, we see by Fact \ref{fact:g00} that $G^{00}(\cU)=G(\cU)$.

    On the other hand, for example
    \[
        [x,y,z]\cdot [0,0,1]=[x,x+y,z+1]\not\equiv [x,y,z],
    \]
    so $G^{00}(\cU)=G(\cU)$ is not the right stabilizer of $p$, so $p$ is a left $\dfg$ type which is not right $f$-generic. (Symmetrically, $p^{-1}=\tp([-x,-y+xz,-z]/\cU)$ is a right $\dfg$ type which is not left $f$-generic.)

Note that $p*p^{-1}$ is $\tp([x,y-cx,z]/\cU)$ where $x\gg y\gg z\gg c\gg \cU$, in particular, this is a global $\emptyset$-definable, $G(\cU)$ (left and right) invariant type.
\end{example}

\begin{remark}
    The situation here is not dissimilar to the classical context --- for an amenable locally compact group $G$, in general, there is always a (left and right) invariant mean on $L^\infty(G)$, but in general (for example, for the Heisenberg group, as for any connected nilpotent Lie group with a non-compact conjugacy class), there are in $L^\infty(G)$ left-invariant means which are not right-invariant. For more details, see \cite{pier_amenable}, in particular Theorem~4.19 and Proposition~22.15.
\end{remark}

\section{Retraction and inversion}

By Chernikov and Simon's resolution to the Newelski-Pillay conjecture (Fact~\ref{fact:NC}), we know that the Ellis subgroups of $S_{G}^{\fs}(\cU,M)$ are also isomorphic to $G(\cU)/G^{00}(\cU)$ in the definably amenable NIP setting. Thus, the invariant Ellis subgroups and the finitely satisfiable Ellis subgroups are \emph{abstractly isomorphic}. There is a \emph{mysterious retraction}\footnote{We quote Simon, \guillemotleft It is still slightly mysterious why such a retraction exists, but it turns out to be rather useful.\guillemotright } map which maps invariant types to finitely satisfiable types in NIP theories. One might (we did) conjecture that this retraction map induces an isomorphism on Ellis subgroups. This conjecture is a little too simple, too na\"{i}ve. In this section, we do the following:
\begin{enumerate}
\item We first recall Simon's retraction map, introduce a \emph{inverted variant} of said map, and prove serval results. In particular, we characterize when the retraction is an isomorphism and when the inverted retraction is an anti-isomorphism.
\item We then restrict to several cases of interest: abelian, fsg, and dfg.
\begin{enumerate}
\item In the abelian case, the retraction and inverted retraction witness isomorphisms, provided the small model contains enough coset representatives.
\item In the $\fsg$ case, the retraction map is the identity on the minimal left ideal (and thus trivially witnesses an isomorphism).
\item In the $\dfg$ case, the inverted retraction map witnesses an anti-isomorphism of Ellis subgroups. Precomposing with group inversion gives an isomorphism of compact topological groups.
\end{enumerate}
\end{enumerate}
The next section is dedicated to \emph{limiting examples}, or examples which limit extending these the above theorems to broader contexts.

Throughout this entire section, we will assume that our underlying theory $T$ is NIP.

\begin{remark} We are interested in the intrinsic relation between invariant Ellis subgroups and finitely satisfiable Ellis subgroups. Thus the retraction map (and its inversion) are obvious choices of maps to consider. Using Theorem~\ref{theorem:main} and Fact~\ref{fact:NC}, for any Ellis groups $E'$ in $S_G^\inv(\cU,M)$ and $E$ in $S_G^\fs(\cU,M)$, we can define an isomorphism by composing the isomorphisms $E'\to G/G^{00}$ and $G/G^{00}\to E$. There are also other, slightly less obvious choices: If $t$ is an idempotent strong right $f$-generic and $u$ is a minimal idempotent in $S_{G}^{\fs}(\cU,M)$, then -- again by Theorem~\ref{theorem:main} and Fact~\ref{fact:NC} -- the formula $p\mapsto t*p(=t*p*t)$ defines an explicit group isomorphism between the Ellis groups $u * S_{G}^{\fs}(\cU,M) * u$ and $t *S^{\inv}_{G}(\cU,M)$, with inverse
$u * F_{M}(-)|_{t *S^{\inv}_{G}(\cU,M)} * u: t *S^{\inv}_{G}(\cU,M) \to u * S_{G}^{\fs}(\cU,M) * u$ (where $F_M$ is the retraction defined in Definition~\ref{dfn:retract}). However, this map is as useful as an abstract isomorphism and tells us nothing about how these subgroups relate to one another as subsets of the space of invariant types. Moreover, these isomorphisms are not continuous, unlike the retraction and its inverted variant we consider below.

On the other hand, if we do not assume NIP and definable amenability, then we lose access to Theorem~\ref{theorem:main} and Fact~\ref{fact:NC}, and it makes sense to ask whether (or to what extent) these formulas still work.
\end{remark}

\begin{question}
    Let $G$ be a definable group, and let $t\in S^{\inv}_{G}(\cU,M)$ and $u\in S^{\fs}_{G}(\cU,M)$, write $E'$ for $t*S^{\inv}_{G}(\cU,M)*t$ and $E$ for $u* S^{\fs}_{G}(\cU,M)*u$. Then consider the functions:
    \begin{itemize}
        \item
        $E\to E'$ given by $p\mapsto t*p*t$,
        \item
        $E'\to E$ given by $p\mapsto u*F_M(p)*u$ (assuming NIP; here $F_M$ is given by Definition~\ref{dfn:retract}).
    \end{itemize}
    Given that for definably amenable NIP groups, $E$ and $E'$ via these functions (and they are inverse to each other, for any choice of $t$ and $u$), what are some other reasonable assumptions (about $T,G,t$, and $u$) that are sufficient for these functions to be e.g.\ injective, surjective, homomorphic, inverse to one another?
\end{question}
Note that in general, even assuming abelianity and existence of invariant definable types (in particular, dfg and definable extreme amenability), the second function may not be well-defined, and the first function may not be injective, see Corollary~\ref{cor:different}.

\subsection{The retraction map: the basics} We briefly recall the construction of Simon's retraction map which was first defined in \cite{simon2015invariant}. We refer the reader to \cite{chernikov2014external} for an interesting application.

\begin{definition}\label{dfn:retract} Suppose $T$ is NIP. Consider the language $\Lc ' = \Lc  \cup \{\mathbf{P}(x)\}$ where $\mathbf{P}(x)$ is a new unary predicate. Let $M \prec M'$ where $M'$ is $|M|^{+}-$saturated and consider the structure $(M',M)$ in the language $\Lc '$ where the interpretation for $\mathbf{P}(x)$ is $M$. Consider $(M',M) \prec (N',N)$ where $(N',N)$ is a $|M'|^{+}$-saturated extension. Then for every $p \in S_{\bar{x}}^{\inv}(\cU,M)$, we can associated to $p$ a unique type in $S_{\bar{x}}^{\fs}(\cU,M)$, known as the \textbf{retraction} of $p$ to $M$ and denoted by $F_{M}(p)$. Indeed, for any formula $\theta(\bar{x},b) \in \Lc _{\bar{x}}(\cU)$, we have that $\theta(\bar{x},b) \in F_{M}(p)$ if and only if the collection of formulas
\begin{equation*}
    \{\theta(\bar{x},b)\} \cup \mathbf{P}(\bar{x}) \cup p|_{N}(\bar{x}),
\end{equation*}
is consistent. We also define the \textbf{inverted retraction map} $K_{M} := F_{M} \circ ^{-1}$ where $^{-1}: S_{G}^{\inv}(\cU,M) \to S_{G}^{\inv}(\cU,M)$ is the model theoretic inverse map, i.e., see Definition \ref{def:mti}.
\end{definition}

The following can be found in \cite[Lemma 3.1, Lemma 3.7]{simon2015invariant}.

\begin{fact}\label{fact:rest_facts} The retraction map $F_{M}$ from $S_{x}^{\inv}(\cU,M)$ to $S_{x}^{\fs}(\cU,M)$ has the following properties: Let $p,q \in S_{x}^{\inv}(\cU,M)$, then
\begin{enumerate}
    \item $F_{M}(p)|_{M} = p|_{M}$,
    \item $F_{M}$ is continuous,
    \item If $p$ is finitely satisfiable in $M$, then $F_{M}(p) = p$,
    \item For any $M$-definable function $f$, $f(F_{M}(p)) = F_{M}(f(p))$.
    \item If $q$ is finitely satisfiable in $M$, then $F_{M}(q_{x} \otimes p_{y}) = q_{x} \otimes F_{M}(p_{y})$.
\end{enumerate}
\end{fact}

We now give several basic results connecting the Newelski product with the retraction map. The following sequence of lemmas and propositions will help us with more complicated computations in later subsections.

\begin{proposition}\label{prop:basic} Suppose that $p \in S_{x}^{\inv}(\cU,M)$ and $q \in S_{y}^{\inv}(\cU,M)$ such that $p$ is definable. Then $F_{M}(p_x \otimes q_y)|_{M} = (F_{M}(q_{y}) \otimes F_{M}(p_x))|_{M}$.
\end{proposition}
\begin{proof} For any $\theta(x,y) \in \Lc (M)$, notice
\begin{align*}
    \theta(x,y) \in F_{M}(q_y) \otimes F_{M}(p_x) &\Longleftrightarrow \theta(x,y) \in F_{M}(q_y) \otimes p_x \\ &\Longleftrightarrow \theta(x,y) \in p_{x} \otimes F_{M}(q_{y}) \\ &\Longleftrightarrow \theta(x,y) \in p_{x} \otimes q_{y} \\ &\Longleftrightarrow \theta(x,y) \in F_{M}(p_{x} \otimes q_{y}).
\end{align*}
We provide the following basic justifications:
\begin{enumerate}
    \item $F_{M}(p_x)|_M = p_x|_M$ and $\theta(x,y) \in \Lc (M)$.
    \item Definable types commutes with finitely satisfiable types.
    \item $F_{M}(q_y)|_{M} = q_x|_M$ and $\theta(x,y) \in \Lc (M)$.
    \item $F_{M}(p_x \otimes q_y)|_M = p_x \otimes q_y|_M$ and $\theta(x,y) \in \Lc (M)$. \qedhere
\end{enumerate}
\end{proof}

It turns out that the previous lemma is true in a more general setting than the one described above. Originally, we conjectured that as long as $p$ is definable $q$ is invariant, then $F_{M}(p_{x} \otimes q_{y}) = F_{M}(q_{y}) \otimes F_{M}(p_{x})$. Our original proof ran into some technical difficulties. A correct proof of the following was provided to us by Simon in private communication. It relies on the following fact, Exercise~2.74 from \cite{Guide}. For our purposes, we do not need the following fact, but we record the result for completeness.

\begin{fact}[$T$ NIP]\label{fact:coheir} If $p \in S_{x}^{\inv}(\cU,M)$ and $p$ is definable over $M$, then $p|_{M}$ admits a unique global coheir.
\end{fact}

\begin{fact} Let $p,q \in S_{x}^{\inv}(\cU,M)$. If $p$ is definable over $M$, then $F_{M}(p_{x} \otimes q_{y}) = F_{M}(q_{y}) \otimes F_{M}(p_{x})$.
\end{fact}

The next lemma is quite helpful for computations. We will often use it without reference.

\begin{lemma}\label{lem:triv} For any $p,q \in S_{G}^{\inv}(\cU,M)$, we have
\begin{enumerate}
    \item $F_{M}(p^{-1}) = F_{M}(p)^{-1}$.
    \item For any $g \in G(M)$, $F_{M}(g \cdot p) = g \cdot F_{M}(p)$ and $F_{M}(p \cdot g) = F_{M}(p) \cdot g$.
    \item For any $\Lc $-formula $\theta(x,z)$, $b \in \cU^{z}$, we have that $\theta(x,b) \in F_{M}(p*q)$ if and only if $ \theta(x\cdot y, b) \in F_{M}(p_{x}\otimes q_{y})$.
\end{enumerate}
\end{lemma}

\begin{proof} These statements follow from (4) of Fact \ref{fact:rest_facts}. The first two statements are obvious modulo this fact. For the third, let $m: G(\cU) \times G(\cU) \to G(\cU)$ be group multiplication. Then
\begin{align*}
\theta(x,b) \in F_{M}(p * q) &\Longleftrightarrow \theta(x,b) \in F_{M}(m(p_{x} \otimes q_{y})) \\ &\Longleftrightarrow \theta(x,b) \in m(F_{M}(p_{x} \otimes q_{y})) \\
&\Longleftrightarrow \theta(x \cdot y,b) \in F_{M}(p_{x} \otimes q_{y}). \qedhere
\end{align*}
\end{proof}

Recall that $\mathcal{F}_{r}$ denotes the collection of strong right $f$-generic types (Definition \ref{def:f}). The next proposition says that the retraction of the strong right $f$-generic types forms a left ideal. This allows us to find a minimal left ideal of finitely satisfiable types such that each element is the restriction of a strong right $f$-generic.

\begin{proposition} If $q \in S_{G}^{\fs}(\cU,M)$ and $p \in S_{G}^{\inv}(\cU,M)$, then $q * F_{M}(p) = F_{M}(q * p)$. Hence, the retract of any left ideal in $S^\inv_G(\cU,M)$ is a left ideal in $S^\fs_G(\cU,M)$. In particular, if $G$ is definably amenable then $F_{M}(\mathcal{F}_{r})$ is a left ideal in $S_{G}^{\fs}(\cU,M)$.
\end{proposition}

\begin{proof} The first statement follows directly from Fact \ref{fact:rest_facts}(5) and Lemma \ref{lem:triv}(3). The rest follows trivially.
\end{proof}

Which elements are retracted to an idempotent? The final proposition of this section demonstrates that only idempotents are retracted to an idempotent in the definably amenable NIP setting.

\begin{proposition} Suppose that $G(x)$ is definably amenable and $u \in F_{M}(\mathcal{F}_{r})$ is an idempotent. Then every $t \in \mathcal{F}_{r}$ such that $F_{M}(t) = u$ is an idempotent.
\end{proposition}

\begin{proof} Let $t \in \mathcal{F}_{r}$ and $F_{M}(t) = u$. Since $u$ is idempotent, it follows that $\hat{\pi}(u)$ is the identity in $G(\cU)/G^{00}(\cU)$ since $\hat{\pi}$ is a semigroup homomorphism. Since $G^{00}(-)$ is type definable over the empty-set, it follows that $\hat{\pi}(u) = \hat{\pi}(t)$. Since $t$ is a strong right $f$-generic over $M$, it follows by Lemma \ref{lemma:apply} that
\begin{equation*}
t * t = t * \tp(e/\cU) = t,
\end{equation*}
and so $t$ is also idempotent.
\end{proof}

\subsection{When is the retraction an isomorphism?} We now present some results which characterize when the retraction/inverse retraction is an isomorphism/anti-isomorphism for definably amenable NIP groups.

\begin{lemma}\label{lemma:codomain} Suppose that $T$ is NIP and $G(x)$ is definably amenable. Let $I$ be a minimal left ideal contained in $F_{M}(\mathcal{F}_{r})$. Fix an idempotent $u \in I$. Consider $t \in \mathcal{F}_{r}$ such that $F_{M}(t) = u$.
\begin{enumerate}
\item If $\img \left(F_{M}|_{t * S_{G}^{\inv}(\cU,M)} \right) \subseteq u *I $, then $F_{M}|_{t*S_{G}^{\inv}(\cU,M)}$ is an isomorphism of Ellis subgroups.
\item If $t^{-1} \in \mathcal{F}_{r}$ and $\img \left(K_{M}|_{t^{-1} * S_{G}^{\inv}(\cU,M)} \right) \subseteq u *I $, then $K_{M}|_{t^{-1}*S_{G}^{\inv}(\cU,M)}$ is an anti-isomorphism of Ellis semigroups.
\end{enumerate}
\end{lemma}
\begin{proof} If condition $(1)$ above holds, then
\begin{equation*}
\hat{\pi}|_{t * S^{\inv}_{G}(\cU,M)} = (\hat{\pi}|_{u * I}) \circ F_{M}|_{t * S^{\inv}_{G}(\cU,M)},
\end{equation*}
since the image of $\hat{\pi}$ only depends on the type over $M$ and $p|_{M} = F_{M}(p)|_{M}$. Since the maps $\hat{\pi}|_{t * S_{G}^{\inv}(\cU,M)}$ and $\hat{\pi}|_{u * I}$ are group isomorphisms (the former by Theorem \ref{theorem:main} and the latter by Fact \ref{fact:NC}), we conclude that $F_{M}|_{t * S^{\inv}_{G}(\cU,M)}$ is also a group isomorphism.

Similarly, if condition $(2)$ holds, then
\begin{equation*}
\hat{\pi}|_{t^{-1} * S^{\inv}_{G}(\cU,M)} = \inv \circ (\hat{\pi}|_{u * I}) \circ K_{M}|_{t^{-1} * S^{\inv}_{G}(\cU,M)},
\end{equation*}
where $\inv\colon G(\cU)/G^{00}(\cU) \to G(\cU)/G^{00}(\cU) $ is the group inverse map. Hence $K_{M}|_{t^{-1}* S_{G}^{\inv}(\cU,M)}$ can be written as the composition of two isomorphism and an anti-isomorphism and thus is also an anti-isomorphism.
\end{proof}

\begin{lemma}\label{lemma:isoequiv} Suppose that $T$ is NIP and $G(x)$ is definably amenable. Let $I$ be a minimal left ideal contained in $F_{M}(\mathcal{F}_{r})$. Fix an idempotent $u \in I$. Consider $t \in \mathcal{F}_{r}$ such that $F_{M}(t) = u$. Suppose that every coset of $G^{00}(\cU)$ in $G(\cU)$ has a representative in $G(M)$. Then the following are equivalent:
\begin{enumerate}
\item For every $g \in G(M)$, $u \cdot g \in u *I$.
\item $F_M|_{t * S_{G}^{\inv}(\cU,M)}: t * S_{G}^{\inv}(\cU,M) \to u *I$ is an isomorphism of Ellis subgroups.
\end{enumerate}
\end{lemma}
\begin{proof}
$(1) \to (2)$. By Lemma \ref{lemma:codomain}, it suffices to prove that the image of $F_{M}|_{t*S_{G}^{\inv}(\cU,M)}$ is a subset of $u *I$. Fix $p \in S_{G}^{\inv}(\cU,M)$ and choose $g \in G(M)$ such that $\pi(g) = \hat{\pi}(p)$. Then
\begin{equation*}
F_{M}(t * p) = F_{M}(t * \tp(g/\cU)) = F_{M}(t) *\tp(g/\cU) = u \cdot g \in u * I.
\end{equation*}
where the first equality follows from Lemma \ref{lemma:apply} and the second from Lemma \ref{lem:triv}.

$(2) \to (1)$. Notice that if $g \in G(M)$, then
\begin{equation*}
u \cdot g = F_{M}(t) * \tp(g/\cU) = F_{M}(t * \tp(g/\cU)) \in u * I. \qedhere
\end{equation*}
\end{proof}

\begin{remark}\label{rem:isomorphism_shortcut} More generally, Suppose that $T$ is NIP and $G(x)$ is definably amenable. Let $\hat{E}$ be an Ellis subgroup of $S_{G}^{\inv}(\cU,M)$, $E$ be an Ellis subgroup of $S_{G}^{\fs}(\cU,M)$, and $f:\hat{E} \to E$. If the following diagram commutes
\begin{center}
\begin{tikzcd}[cramped, column sep=0]
%[row sep=medium]
\hat{E} \arrow[rr, "f"] \arrow[dr, "\pi|_{\hat{E}}"'] && E \arrow[ld, "\pi|_{E}"] \\
&G(\cU)/G^{00}(\cU) &
\end{tikzcd}
\end{center}

\noindent then $f$ is a group isomorphism.
\end{remark}

\begin{lemma}\label{lemma:antiequiv} Suppose that $T$ is NIP and $G(x)$ is definably amenable. Let $I$ be a minimal left ideal contained in $F_{M}(\mathcal{F}_{r})$. Fix an idempotent $u \in I$. Consider $t \in \mathcal{F}_{r}$ such that $F_{M}(t) = u$. Suppose that every coset of $G^{00}(\cU)$ in $G(\cU)$ has a representative in $G(M)$. Assume moreover that $t^{-1} \in \mathcal{F}_{r}$. Then the following are equivalent.
\begin{enumerate}
\item For every $g \in G(M)$, $g \cdot u \in u *I$.
\item $K_M|_{t^{-1} * S_{G}^{\inv}(\cU,M)}: t ^{-1}* S_{G}^{\inv}(\cU,M) \to u *I$ is an anti-isomorphism of Ellis subgroups.
\end{enumerate}
\end{lemma}

\begin{proof} The proof is similar to the proof of Lemma \ref{lemma:isoequiv}.

$(1) \to (2)$. By Lemma \ref{lemma:codomain}, it suffices to prove that the image of $K_{M}|_{t^{-1}*S_{G}^{\inv}(\cU,M)}$ is a subset of $u *I$. Fix $p \in S_{G}^{\inv}(\cU,M)$ and choose $g \in G(M)$ such that $\pi(g) = \hat{\pi}(p)$. Then
\begin{align*}
K_{M}(t^{-1} * p) &= F_{M}(t^{-1} * \tp(g/\cU))^{-1} = F_{M}( [t^{-1} *\tp(g/\cU)]^{-1}) \\ &= F_{M}(\tp(g^{-1}/\cU) * t) = \tp(g^{-1}/\cU) * F_{M}(t) = g^{-1} \cdot u \in u * I,
\end{align*}
$(2) \to (1)$. Consider the following:
\begin{align*}
g \cdot u &= \tp(g/\cU) * F_{M}(t) = (F_{M}(t)^{-1} * \tp(g^{-1}/\cU))^{-1}\\
 &= (F_{M}(t^{-1} * \tp(g^{-1}/\cU))^{-1}
=K_{M}(t^{-1} * \tp(g^{-1}/\cU)) \in u * I.
\end{align*}
The above sequence of equalities uses several instances of Lemma \ref{lem:triv}.
\end{proof}

\subsection{Abelian NIP groups} We prove that both the retraction and the inverted retraction witness isomorphisms of Ellis subgroups in the abelian NIP setting. We recall that all abelian groups are amenable and thus definably amenable. Hence the collection of global strong right $f$-generic types over $M$ is non-empty.

\begin{proposition}\label{prop:up} Let $t \in S_{G}^{\inv}(\cU,M)$ and suppose that $t$ is idempotent. Then the following are equivalent:
\begin{enumerate}
    \item $t$ is strong left $f$-generic.
    \item $t$ is strong right $f$-generic.
    \item $t^{-1}$ is strong left $f$-generic.
    \item $t^{-1}$ is strong right $f$-generic.
\end{enumerate}
\end{proposition}

\begin{proof} Follows directly from Fact \ref{fact:easy}.
\end{proof}

\begin{theorem}\label{thm:invert_iso_abelian} Suppose that $T$ is NIP, $G(x)$ is abelian, and $G(M)$ contains representatives for each coset of $G^{00}(\cU)$ in $G(\cU)$. Let $I$ be a minimal left ideal of $F_{M}(\mathcal{F}_{r})$. Fix an idempotent $u \in I$. Consider $t \in \mathcal{F}_{r}$ such that $F_{M}(t) = u$. Then
\begin{enumerate}
\item The map $F_{M}|_{t * S^{\inv}_{G}(\cU,M)}: {t * S^{\inv}_{G}(\cU,M)} \to u *I$ is an isomorphism of Ellis subgroups.
\item The map $K_{M}|_{t^{-1} * S^{\inv}_{G}(\cU,M)}: {t^{-1} * S^{\inv}_{G}(\cU,M)} \to u *I$ is an isomorphism of Ellis subgroups.
\end{enumerate}
\end{theorem}

\begin{proof} To prove statement $(1)$, we apply Lemma \ref{lemma:isoequiv}. Since $G(x)$ is abelian and $u$ is idempotent, we have that for any $g \in G(M)$,
\begin{equation*}
g \cdot u = \tp(g/\cU) * u *u = u * \tp(g/\cU) * u \in u * I.
\end{equation*}
To prove statement $(2)$, we first realize that $t^{-1} \in \mathcal{F}_{r}$ by Proposition \ref{prop:up}. Since $G(x)$ is abelian, $G(\cU)/G^{00}(\cU)$ is also abelian. Since any anti-isomorphism between abelian groups is automatically an isomorphism, it suffices to prove that $K_{M}|_{t^{-1} * S^{\inv}(\cU,M)}$ is an anti-isomorphism. To do so, we will apply Lemma \ref{lemma:antiequiv}. Again, since $G(x)$ is abelian and $u$ is idempotent, we have that for any $g \in G(M)$,
\begin{equation*}
u \cdot g = u * u \cdot \tp(g/\cU) = u * \tp(g/\cU) * u \in u * I.
\end{equation*}
Hence, the statement holds.
\end{proof}

\subsection{NIP fsg groups} Here we briefly consider the case of NIP $\fsg$ groups. In this setting, the unique minimal left ideal of the semigroups $S_{G}^{\fs}(\cU,M)$ and $S_{G}^{\inv}(\cU,M)$ coincide -- hence the retraction map restricted to this set is the identity. Thus the retraction map witnesses an isomorphism of Ellis subgroups. We remark that this result is more or less folklore, and can be massaged out by results already present in the literature.

We recall that a group $G(x)$ is $\fsg$ (over $M$) if there exists some $p \in S_{G}^{\inv}(\cU,M)$ such that $p$ is left $\fsg$. We recall that $\fsg$ does not dependent on the choice of the small model $M$ (\cite[Remark 4.4]{NIP1}).

\begin{proposition}\label{prop:fsgmain} Suppose that $G(x)$ is $\fsg$. Then $\mathcal{F}_{r}$ is the unique minimal left ideal in $S_{G}^{\fs}(\cU,M)$. As consequence, the retraction map is the identity on $\mathcal{F}_{r}$ and thus is an isomorphism when restricted to Ellis subgroups.
\end{proposition}
\begin{proof} If $G(x)$ has $\fsg$, then $\mathcal{F}_{r}$ is precisely the set of types which are $\fsg$ over $M$ (e.g., see \cite[Corollary 3.5]{hrushovski2012note} + \cite[Proposition 4.2]{NIP1}). Thus $\mathcal{F}_{r} \subseteq S_{G}^{\fs}(\cU,M)$. Since $\mathcal{F}_{r}$ is a two-sided ideal of $S_{G}^{\inv}(\cU,M)$, it is also a two-sided ideal in $S_{G}^{\fs}(\cU,M)$. Also, $\mathcal{F}_{r}$ remains a minimal left ideal. Notice that if $w \in \mathcal{F}_{r}$, then $\mathcal{F}_{r} * w = \mathcal{F}_{r}$ and so $\mathcal{F}_{r}$ cannot properly contain any left ideals. By Fact \ref{fact:equiv}, $\mathcal{F}_{r}$ is the unique minimal left ideal in $S_{G}^{\fs}(\cU,M)$ (A more direct argument is given by \cite[Theorem 3.8]{pillay2013topological}\footnote{The language is somewhat different from the one presented in this article. In \cite{chernikov2014external}, the authors explain, ``In \cite{pillay2013topological} it was proved that when $G$ has $\fsg$ then there is a unique minimal closed $G(M)$-invariant subspace of $S_{G}(M^{ext})$ and it coincides with the set of generic types''. This is equivalent to the fact that the $\fsg$ types form the unique minimal left ideal in $S_{G}^{\fs}(\cU,M)$.}). By Fact \ref{fact:rest_facts}(3), the map $F_{M}|_{\mathcal{F}_{r}} = \id_{\mathcal{F}_{r}}$ and thus the statement holds.
\end{proof}

\subsection{NIP dfg groups} We now consider NIP $\dfg$ groups. This case is quite interesting and is the namesake of our paper. In this subsection, we will show that if $G(x)$ is $\dfg$ over $M$ (i.e., $G(x)$ admits a $\dfg$ type over $M$), then the inverse retraction is an anti-isomorphism restricted of certain Ellis subgroups. If one wishes, they can easily turn this anti-isomorphism into an \emph{honest-to-goodness} isomorphism via precomposition with group inversion. In the next section, we will see Example~\ref{example:dfg}, where this is essentially the best one can hope for (i.e., the retraction map does not even map invariant Ellis subgroups to finitely satisfiable Ellis subgroups).

Unlike the $\fsg$ case, it is not true that if $G(x)$ is $\dfg$ over $M$ then every strong right $f$-generic type is also $\dfg$ over $M$ (e.g., see \cite[Lemma 4.10]{yao2023minimal}). Thus the collection of right $\dfg$ types over $M$ may be a proper subset of $\mathcal{F}_{r}$. We thank Ningyuan Yao for pointing this out to us. We denote the collection of right $\dfg$ types over $M$ as $R_{\dfg}$

We begin with some basic observations regarding $\dfg$ types. Throughout this entire subsection, we assume that $T$ is NIP and $G(x)$ is $\dfg$ over $M$.

\begin{lemma}\label{lemma:dfg1} Suppose that $p \in S_{G}^{\inv}(\cU,M)$. If $p$ is right $\dfg$ over $M$, then $p * q$ is right $\dfg$ over $M$. As consequence, $R_{\dfg}$ forms a right subideal of $\mathcal{F}_{r}$.
\end{lemma}

\begin{proof} Since the strong right $f$-generic types form a two-sided ideal (i.e., see Proposition \ref{prop:unique_min} and Fact \ref{fact:equiv}), it follows that the type $p * q$ is a strong right $f$-generic. Hence it suffices to prove that $p * q$ is $M$-definable. Choose a model $N$ such that $M \prec N \prec \cU$ and $N$ is $|M|^{+}$-saturated. Notice that for any $\Lc $-formula $\theta(x,z)$ and $a \in \cU^{z}$ such that $\theta(x,a) \vdash G(x)$, if $b \models q|_{N}$ and $a_0 \in N^{z}$ such that $a \equiv_{M} a_0$, we have that
\begin{align*}
    \theta(x,a) \in (p * q) &\Longleftrightarrow \theta(x,a_0) \in (p* q) \Longleftrightarrow \theta(x \cdot b,a_0) \in p \\
    &\Longleftrightarrow \theta(x,a_0) \in p \cdot b \Longleftrightarrow \models d^{\theta}_{p \cdot b }(a_0) \Longleftrightarrow \models d^{\theta}_{p \cdot b}(a).
\end{align*}
The formula $d_{p \cdot b}^{\varphi}(y)$ is an $\Lc (M)$-formula since $p$ is right $\dfg$ over $M$ and thus the last line is justified.
\end{proof}

\begin{lemma}\label{lemma:compute} Suppose that $p,q \in S_{G}^{\inv}(\cU,M)$ and $p$ is right $\dfg$. Then
\begin{equation*}
    F_{M}(p * q) = (F_{M}(q)^{-1} * F_{M}(p)^{-1})^{-1}.
\end{equation*}
\end{lemma}

\begin{proof} Since $p$ is right $\dfg$ over $M$, Lemma \ref{lemma:dfg1} implies that $p* q$ is $M$-definable. Since $T$ is NIP, the type $(p*q)|_{M}$ admits a unique global coheir (i.e., see Fact \ref{fact:coheir}). Both $F_{M}(p * q)$ and $(F_{M}(q)^{-1} * F_{M}(p)^{-1})^{-1}$ are both global types which are finitely satisfiable in $M$ and so it suffices to show,
\begin{equation*}
    F_{M}(p * q)|_{M} = (p * q)|_{M} = (F_{M}(q)^{-1} * F_{M}(p)^{-1})^{-1}|_{M}.
\end{equation*}
Suppose that $\varphi(x,b) \in \Lc _{x}(M)$ and $\varphi(x,b) \vdash G(x)$. Applying Proposition \ref{prop:basic}, Lemma \ref{lem:triv}, and some basic deductions completes the proof. Indeed, notice
\begin{align*}
    \varphi(x,b) \in F_{M}(p * q) &\Longleftrightarrow \varphi(x \cdot y,b) \in F_{M}(p_{x} \otimes q_{y}) \\
    &\Longleftrightarrow \varphi(x \cdot y,b) \in p_{x} \otimes q_{y} \Longleftrightarrow \varphi(x,b) \in p * q.
\end{align*}
Likewise,
\begin{align*}
    \varphi(x,b) \in (F_{M}(q)^{-1} * F_{M}(p)^{-1})^{-1} &\Longleftrightarrow \varphi(x^{-1},b) \in (F_{M}(q)^{-1} * F_{M}(p)^{-1}) \\
    &\Longleftrightarrow \varphi((x \cdot y)^{-1},b) \in (F_{M}(q_{x})^{-1} \otimes F_{M}(p_{y})^{-1}) \\
    &\Longleftrightarrow \varphi((x \cdot y)^{-1},b) \in (F_{M}(q_{x}^{-1}) \otimes F_{M}(p_{y}^{-1})) \\
    &\overset{(*)}{\Longleftrightarrow} \varphi((x \cdot y)^{-1},b) \in F_{M}(p_{y}^{-1} \otimes q_{x}^{-1})\\
    &\Longleftrightarrow\varphi( (x \cdot y)^{-1} ,b) \in ( p_{y}^{-1} \otimes q_{x}^{-1})\\
    &\Longleftrightarrow \varphi( y^{-1} \cdot x^{-1} ,b) \in ( p_{y}^{-1} \otimes q_{x}^{-1})\\
    &\Longleftrightarrow \varphi( y \cdot x ,b) \in ( p_{y} \otimes q_{x})\\
    &\Longleftrightarrow \varphi(x ,b) \in ( p * q).
\end{align*}
Equation $(*)$ follows from Proposition \ref{prop:basic}.
\end{proof}

\begin{proposition}\label{prop:ideal} Suppose that $R$ is a right subideal of $R_{\dfg}$ in $S_{G}^{\inv}(\cU,M)$. Then $F_{M}(R^{-1})$ is a left subideal of $F_{M}(R_{\dfg}^{-1})$ in $S_{G}^{\fs}(\cU,M)$. In particular, $F_{M}(R_{\dfg}^{-1})$ is a left ideal of $S_{G}^{\fs}(\cU,M)$.
\end{proposition}

\begin{proof} Suppose that $p \in F_{M}(R)$ and $q \in S_{G}^{\fs}(\cU,M)$. It suffices to check that $q * p^{-1} \in F_{M}(R^{-1})$. We let $\hat{p}$ be the unique definable extension of $p|_{M}$, which we note must be in $R$. Notice that
\begin{align*}
    q * p^{-1} = F_{M}(q) * F_{M}(\hat{p}^{-1}) &= \left( \left(F_{M}(q^{-1})^{-1} * F_{M}(\hat{p})^{-1} \right)^{-1} \right)^{-1} \\
    &= (F_{M}(\hat{p} * q^{-1}))^{-1}
\end{align*}
The first equality follows since $F_{M}$ is constant on finitely satisfiable types. The second equality is replacing types with their double involutions. The final equality is an application of Lemma \ref{lemma:compute}.
\begin{enumerate}
    \item Since $p \in F_{M}(R)$, we have $\hat{p} \in R$.
    \item Since $\hat{p} \in R$ and $R$ is a right ideal, we have $\hat{p} * q^{-1} \in R$.
    \item Since $\hat{p} * q^{-1} \in R$, we have $(\hat{p} * q^{-1})^{-1} \in R^{-1}$.
    \item Thus $(F_{M}(\hat{p} * q^{-1}))^{-1} \in F_{M}(R^{-1})$ and by the equation above, $q * p^{-1} \in F_{M}(R^{-1})$.
\end{enumerate}
Hence $F_{M}(R^{-1})$ is a left ideal of $S_{G}^{\fs}(\cU,M)$.
\end{proof}

\begin{theorem}\label{theorem:map} Suppose that $T$ is NIP, $G(x)$ is $\dfg$, and $t$ is a right $\dfg$ type over $M$. Then $K_{M}|_{t*S_{G}^{\inv}(\cU,M)}$ is an anti-isomorphism from an invariant Ellis subgroup to a finitely satisfiable Ellis subgroup.
\end{theorem}

\begin{proof} Without loss of generality, suppose that $t$ is an idempotent. Let $R := t *S_{G}^{\inv}(\cU,M)$. By Proposition \ref{prop:unique_min} and Fact \ref{fact:equiv}, $R$ is a both a minimal right ideal and an Ellis subgroup of $S_{G}^{\inv}(\cU,M)$. By Proposition \ref{prop:ideal}, $K_{M}(R)$ is a left ideal in $S_{G}^{\fs}(\cU,M)$. We first argue that $K_{M}(R)$ is minimal. Suppose that $L \subseteq K_{M}(R)$ is a minimal left ideal and let $u \in L$ be an idempotent. Since $u$ is in the image of $K_{M}|_{R}$, there exists some element of $R$ which is mapped to $u$. Furthermore, we recall that $F_{M}(-)|_{M} = -|_{M}$ and $G^{00}(-)$ is $\emptyset$-definable. Thus the element which is mapped to $u$ under $K_{M}$ must also imply $G^{00}(x)$. Note that $t$ is the only $t * S^{\inv}(\cU,M)$ which implies $G^{00}(x)$ [since $\hat{\pi}_{t *S^{\inv}(\cU,M)}: t * S^{\inv}(\cU,M) \to G(\cU)/G^{00}(\cU)$ is an isomorphism]. Thus $K_{M}(t) = u$

We will now show that $L = K_{M}(R)$. Let $p = F_{M}(q^{-1})$ for some $q \in R$. Then
\begin{equation*}
p = F_{M}(q)^{-1} = F_{M}(t * q)^{-1} = F_{M}(q^{-1}) * F_{M}(t^{-1}) = p * u \in L,
\end{equation*}
where the second equation follows from the fact that $t,q \in R$ and $t$ is the identity element of the group, the third equation follows from Lemma \ref{lemma:compute}, and $p * u \in L$ because $L$ is a left ideal. Thus $K_{M}(R) \subseteq L$ and so $K_{M}(R)$ must be minimal.

Thus $K_{M}(R)$ is a minimal left ideal in $S_{G}^{\fs}(\cU,M)$. It is well known that this is a group (i.e., see \cite[Proposition 5.6]{chernikov2014external}). This also follows from the observation that $K_{M}(R)$ is the union of Ellis subgroups, but contains only one idempotent (note that if it contained 2 or more idempotents, then $R$ must also contain 2 or more). It is clear from Lemma \ref{lemma:compute} that the statement holds.
\end{proof}

\begin{remark}[NIP] We remark that if $p \in S_{G}^{\inv}(\cU,M)$ is a right $f$-generic type and $p \cdot G(\cU)$ is a closed subset of $S_{G}^{\inv}(\cU,M)$, then $F_{M}(Gp^{-1})$ is a minimal left ideal and Ellis subgroup of $S_{G}^{\fs}(\cU,M)$. This statement does not use the $\dfg$ hypothesis. The proof is similar to the one above.
\end{remark}

\begin{remark}
    Every closed, left $G(M)$-invariant subset of $S^\fs_G(\cU,M)$ is a left ideal.
\end{remark}

More generally, we can prove the following.

\begin{theorem}[NIP]
    \label{thm:main_dfg_alternative}
    Suppose $p\in S^\inv_G(\cU,M)$ is a right $f$-generic type. Consider the following conditions:
    \begin{enumerate}
        \item
        \label{it:pG_closed}
        $p\cdot G(\cU)$ is closed,
        \item
        \label{it:orbit_closure_contained}
        $\overline{F_M(p)\cdot G(M)}\subseteq F_M(p\cdot G(\cU))$ in $S_G^\fs(\cU,M)$,
        \item
        \label{it:orbit_closure_equal}
        $\overline{F_M(p)\cdot G(M)}=F_M(p\cdot G(\cU))$ in $S_G^\fs(\cU,M)$,
        \item
        \label{it:image_closed}
        $F_M(p\cdot G(\cU))$ is closed in $S_G^\fs(\cU,M)$,
        \item
        \label{it:image_ideal}
        $F_M(G(\cU)\cdot p^{-1})$ is a left ideal in $S_G^\fs(\cU,M)$,
        \item
        \label{it:image_min_ideal}
        $F_M(G(\cU))\cdot p^{-1})$ is a minimal left ideal and an Ellis subgroup in $S_G^\fs(\cU,M)$,
        \item
        \label{it:anti_hom}
        for all $q\in S^\inv_G(\cU,M)$, we have $F_M(p*q)^{-1}=F_M(q)^{-1}*F_M(p)^{-1}$,
        \item
        \label{it:image_ellis}
        $F_M(G(\cU)\cdot p^{-1})$ is an Ellis subgroup in $S_G^\fs(\cU,M)$,
        \item
        \label{it:standard_antiiso}
        $K_M|_{p\cdot G(\cU)}$ is an anti-isomorphism with an Ellis subgroup in $S_G^\fs(\cU,M)$.
    \end{enumerate}
    Then:
    \begin{itemize}
        \item
        \eqref{it:pG_closed} implies all the others,
        \item
        the conditions \eqref{it:orbit_closure_contained}-\eqref{it:anti_hom} are equivalent and imply \eqref{it:image_ellis}-\eqref{it:standard_antiiso},
        \item
        \eqref{it:image_ellis} and \eqref{it:standard_antiiso} are equivalent.
    \end{itemize}
\end{theorem}
\begin{proof}
    Note that \eqref{it:pG_closed} trivially implies \eqref{it:image_closed}, and \eqref{it:image_min_ideal} trivially implies \eqref{it:image_ellis}. Since $K_M(p\cdot G(\cU))=F_M(G(\cU)\cdot p^{-1})$, \eqref{it:standard_antiiso} clearly implies \eqref{it:image_ellis}, and the converse follows from this and Remark~\ref{rem:isomorphism_shortcut}.

    It remains to show that the conditions \eqref{it:orbit_closure_contained}-\eqref{it:anti_hom} are all equivalent.
    To streamline the argument, let us state the following simple claim.
    \begin{claim-star}
        $F_M$ is bijective on $p\cdot G(\cU)$, $\pi$ is bijective on $p\cdot G(\cU)$ and $F_M(p\cdot G(\cU))$ and their inverses, and $\pi$ is surjective when restricted to $\overline{F_M(p)\cdot G(M)}$ (closure in $S_G^\fs(\cU,M)$) and its inverse.
    \end{claim-star}
    \begin{clmproof}
        Bijectivity of $F_M$ and $\pi$ on the described sets follows easily from Theorem~\ref{theorem:main} and the fact that $\pi\circ F_M=\pi$. For surjectivity on $\overline{F_M(p)\cdot G(M)}$, just note that its inverse is $\overline{G(M)\cdot F_M(p)^{-1}}$, which is a left ideal in $S_G^\fs(\cU,M)$.
    \end{clmproof}

    Suppose \eqref{it:orbit_closure_contained} holds. Then by Claim, $\pi$ is bijective when restricted to $\overline{F_M(p)\cdot G(M)}$ and injective on its superset $F_M(p\cdot G(\cU))$, which trivially implies \eqref{it:orbit_closure_equal}.

    \eqref{it:orbit_closure_equal} trivially implies \eqref{it:image_closed}, and \eqref{it:image_closed} trivially implies \eqref{it:orbit_closure_contained}, as $F_M(p\cdot G(M))=F_M(p)\cdot G(M)$, so we have that these three conditions are equivalent.

    Suppose \eqref{it:orbit_closure_contained}-\eqref{it:image_closed} hold. Then by \eqref{it:orbit_closure_equal} we have that $F_M(G(\cU)\cdot p^{-1})=\overline{G(M)\cdot F_M(p)^{-1}}$, which is a left ideal, yielding \eqref{it:image_ideal}.

    Assume \eqref{it:image_ideal}. Then $F_M(G(\cU)\cdot p^{-1})$ contains a minimal left ideal $L$, which contains an Ellis subgroup $E$. Note that $\pi|_E$ is bijective and $\pi|_{F_M(G(\cU)\cdot p^{-1})}$ is injective by Claim, so necessarily $E=L=F_M(G(\cU)\cdot p^{-1})$, yielding \eqref{it:image_min_ideal}.

    Now, suppose \eqref{it:image_min_ideal} holds. Then $F_M(G(\cU)\cdot p^{-1})$ is closed (as a minimal left ideal), which gives us \eqref{it:image_closed} (and hence also \eqref{it:orbit_closure_contained} and \eqref{it:orbit_closure_equal}).

    Finally, since by Theorem~\ref{theorem:main}, $p\cdot G(\cU)$ is right ideal in $S^\inv_G(\cU,M)$, \eqref{it:anti_hom} trivially implies \eqref{it:image_ideal}. For the converse, assuming \eqref{it:image_ideal}, note that \[\pi(F_M(p*q)^{-1})=\pi(q)^{-1}\pi(p)^{-1}=\pi(F_M(q)^{-1}*F_M(p)^{-1}),\]
    and we have
    $F_M(p*q)^{-1}\in F_M(p\cdot G(\cU))^{-1}$
     and by \eqref{it:image_ideal}, $F_M(q)^{-1}*F_M(p)^{-1}\in F_M(p\cdot G(\cU))^{-1}$ (because it is an ideal). Since $\pi$ is injective on $F_M(p\cdot G(\cU))^{-1}$ by Claim, \eqref{it:anti_hom} follows.
\end{proof}

\begin{remark}
    \begin{itemize}
        \item
        The condition \eqref{it:pG_closed} (and hence all the others also) above holds in particular when $p$ is dfg or $G/G^{00}$ is finite.
        \item
        The conditions \eqref{it:pG_closed} and \eqref{it:image_closed} roughly correspond to the CIG1 (``admits compact ideal groups'') property introduced by the first author and Artem Chernikov in \cite{chernikov2023definable}.
        \item
        The types satisfying \eqref{it:pG_closed} form a right ideal of $S^\inv_G(\cU,M)$ containing $R_\dfg$, and many results of this section involving $R_\dfg$ remain true if we replace all mentions of $R_\dfg$ by this ideal.
        \item
        If $G(M)$ realises all cosets, then in all the conditions, we can replace $p\cdot G(\cU)$ with $p\cdot G(M)$, $F_M(p\cdot G(\cU))$ with $F_M(p)\cdot G(M)$, and likewise with the inverses. In particular, \eqref{it:pG_closed} just says that $p\cdot G(M)$ is closed, and \eqref{it:orbit_closure_contained}-\eqref{it:image_min_ideal} just say that $F_M(p)\cdot G(M)$ is closed or $G(M)\cdot F_M(p^{-1})$ is a minimal left ideal.
        \item
        We do not know whether \eqref{it:pG_closed} is actually equivalent to \eqref{it:image_closed}.
        \item
        If $G$ is an fsg group, then \eqref{it:pG_closed} is equivalent to \eqref{it:image_closed}, and true if and only fsg types form a single Ellis subgroup.
        % (this follows from Theorem~\ref{...} and the fact that in an fsg group, left and right $f$-generics are the generic types)
        In particular, $S^1$ shows that \eqref{it:image_ellis}-\eqref{it:standard_antiiso} do not imply any of the preceding conditions.
        \item
        The proof of Theorem~\ref{thm:main_dfg_alternative} does not rely on Fact~\ref{fact:NC}
    \end{itemize}
\end{remark}

\begin{corollary} Suppose that $T$ is NIP, $G(x)$ is $\dfg$, and $t$ is a right $\dfg$ type over $M$. Then $K_{M}|_{t*S_{G}^{\inv}(\cU,M)} \circ \inv$ is an isomorphism from an invariant Ellis subgroup to a finitely satisfiable Ellis subgroup where $\inv$ is the group inversion map.
\end{corollary}

\begin{proof} Clear.
\end{proof}

We actually have a homeomorphism between Ellis subgroups [with the induced topology].

\begin{corollary}\label{cor:isocompact} Suppose that $T$ is NIP, $G(x)$ is $\dfg$, and $t$ is a right $\dfg$ type over $M$.\footnote{The conclusion of this corollary also holds when $t$ is a type satisfying \eqref{it:pG_closed} in Theorem~\ref{thm:main_dfg_alternative}.}
% Sketch of this: suppose 4.24(1) holds, we want to show that the multiplication is continuous in the right argument in the inverse retracted group in fs. Every element there is of the form K_M(p) with p in the original group. Take any net K_M(p_i) convergent to K_M(p).
% By (1) we know that K_M is a homeomorphism on the original Ellis group, so then p_i converges to p. Now taking any K_M(q) we have that p_i*q converges to p*q (by left continuity in inv), and by 4.24(7) we have that K_M(q)*K_M(p)=K_M(p*q) and K_M(q)*K_M(p_i) = K_M(p_i*q). Passing to the limit with the right hand side we get the conclusion (separate continuity). Then apply Ellis and compactness.
Then
\begin{enumerate}
\item Both $t * S_{G}^{\inv}(\cU,M)$ and $K_{M}(t * S_{G}^{\inv}(\cU,M))$ are compact Hausdorff groups with the induced topology from the type space.
\item The map $K_{M}\circ \inv$, where $\inv$ is the group-theoretic inversion in the Ellis group ${t * S_{G}^{\inv}(\cU,M)}$, is an isomorphism of topological groups.
\end{enumerate}
\end{corollary}

\begin{proof} Every element of $R:=t * S_{G}^{\inv}(\cU,M)$ is definable (i.e., see Lemma \ref{lemma:dfg1}). By Fact \ref{fact:nb}, the set $t * S_{G}^{\inv}(\cU,M)$ is closed/compact with the induced topology and moreover the Newelski product is separately continuous when restricted to $t * S_{G}^{\inv}(\cU,M)$. Thus $(R,*)$ forms a separately continuous, compact Hausdorff group and by the Ellis joint continuity theorem \cite{ellis1957locally}, we conclude that $(R, *)$ is an \emph{honest-to-goodness} compact Hausdorff group.

Since $F_{M}$ and the model theoretic inversion maps are continuous, we have that $K_{M}$ is also continuous. Since the image of a compact set is compact, we see that $K_{M}(R)$ is compact and thus -- noting that the group-theoretic inversion $\inv\colon R\to R$ is continuous because $R$ is a topological group -- the bijection $K_{M} \circ \inv$ gives a homeomorphism between $R$ and $K_{M}(R)$ . Thus $K_{M}(R)$ with the induced topology also forms a compact Hausdorff group and $K_{M} \circ \inv$ is a topological group isomorphism.
\end{proof}

\subsection{Abelian NIP dfg groups} In this subsection, we make a quick remark regarding abelian NIP $\dfg$ groups. In fact, this work was in some sense inspired by the example $M = (\mathbb{Z};0,1,+,<)$. In this example, it is clear that the retraction map is an isomorphism of Ellis subgroups. However, it does not fit into the picture in the abelian NIP group section (since $G(M)$ does not contain enough coset representatives).

\begin{proposition}
\label{prop:abelian_dfg}
Suppose that $T$ is NIP, $G(x)$ is abelian and $\dfg$, and $t$ is a right $\dfg$ type over $M$. Then
\begin{enumerate}
\item If $p, q \in S_{G}^{\inv}(\cU,M)$ and $p$ is right $\dfg$, then $F_{M}(p * q) = F_{M}(q) * F_{M}(p)$.
\item If $R \subseteq R_{\dfg}$ is a right ideal, then $F_{M}(R)$ is a left ideal in $S_{G}^{\fs}(\cU,M)$.
\item $F_{M}|_{t *S_{G}^{\inv}(\cU,M)}$ is an isomorphism of Ellis semigroups (also a homeomorphism).
\end{enumerate}
\end{proposition}

\begin{proof} We prove the statements:
\begin{enumerate}
\item Similar to the proof of Lemma \ref{lemma:compute}. Suppose that $\varphi(x,b) \in \Lc _{x}(M)$. Similar justifications give the following sequence of bi-implications:
\begin{align*}
    \varphi(x,b) \in F_{M}(p * q) &\Longleftrightarrow \varphi(x \cdot y,b) \in F_{M}(p_{x} \otimes q_{y}) \\
    &\Longleftrightarrow \varphi(x \cdot y,b) \in F_{M}(q)_{y} \otimes F_{M}(p)_{x} \\
    &\Longleftrightarrow \varphi(y \cdot x,b) \in F_{M}(q)_{y} \otimes F_{M}(p)_{x} \\
    &\Longleftrightarrow \varphi(x \cdot y,b) \in F_{M}(q)_{x} \otimes F_{M}(p)_{y} \\
    &\Longleftrightarrow \varphi(x,b) \in F_{M}(q) * F_{M}(p).
\end{align*}
A similar argument completes the proof.
\item Follows directly from Statement (1).
\item Similar to Theorem \ref{theorem:map} and Corollary \ref{cor:isocompact}. \qedhere
\end{enumerate}
\end{proof}

\subsection{Keisler measures} As stated in the introduction, the left-right phenomena has been implicitly observed in the setting of Keisler measures. Here we take the opportunity to make this observation explicit. We recall that if $T$ is NIP, then the space of global $M$-invariant Keisler measures, $\mathfrak{M}_{G}^{\inv}(\cU,M)$, as well as finitely satisfiable, denoted $\mathfrak{M}_{G}^{\fs}(\cU,M)$ which concentrate on $G$ form a left continuous compact Hausdorff semigroups under the definable convolution operation. We still use the symbol `$*$' to denote this operation and refer the reader to \cite{chernikov2022definable,chernikov2023definable} for further reading.

\begin{remark} Suppose that $T$ is NIP and $G(x)$ is definably amenable. Then \cite[Theorem 5.1]{chernikov2023definable} gives the following:
\begin{enumerate}
    \item The unique minimal left ideal of $(\mathfrak{M}_{G}^{\inv}(\cU,M),*)$ is the collection of $G(\cU)$-right-invariant measures. Call this minimal left ideal $\mathcal{J}$. The Ellis subgroups of this minimal left ideal are isomorphic to the trivial group.
    \item If $\nu \in \mathfrak{M}_{G}^{\fs}(\cU,M)$ and $\nu$ is $G(M)$-left-invariant, then $\{v\}$ is a minimal left ideal of $(\mathfrak{M}_{G}^{\fs}(\cU,M),*)$. We let $\mathcal{H}$ be the union of minimal left ideals.
\end{enumerate}
The retraction map on types extends to a retraction map on measures via the push-forward. The inversion map is defined similarly and thus $K_{M}$ is still well-defined.
\end{remark}

\begin{proposition}\label{prop:implicit} The map $K_{M}|_\mathcal{J}$ is an anti-homomorphism from $\mathcal{J}$ to $\mathcal{H}$.
\end{proposition}

\begin{proof} Notice that if $\mu \in \mathfrak{M}_{G}^{\inv}(\cU,M)$ and $\mu$ is $G(\cU)$-right-invariant then $\mu^{-1}$ is $G(\cU)$-left-invariant and $F_{M}(\mu^{-1})$ remains $G(M)$-left-invariant. Thus $K_{M}(J) \subseteq \mathcal{H}$. Note that $K_{M}$ is a homomorphism since
\begin{equation*}
    F_{M}(\mu * \nu) = F_M(\mu) = F_{M}(\nu) * F_{M}(\mu).
\end{equation*}
The first equality follows from the fact that the Ellis subgroups of $\mathcal{J}$ is isomorphic to the trivial group while the second equality follows from the fact that $\{F_{M}(\mu)\}$ is a minimal left ideal of $\mathfrak{M}_{G}^{\fs}(\cU,M)$.
\end{proof}

\section{Limiting Examples} The purpose of this section is to provide several limiting examples. We first argue that the retraction map need not restrict to an isomorphism of Ellis subgroups (in fact, it may not even map Ellis subgroups to groups). The example we provide is NIP $\dfg$. We then argue that there exists groups such that the inverted retraction map need not restrict to anti-isomorphisms of Ellis subgroups; again, this map does not map Ellis subgroups to groups. Here, the example is NIP $\fsg$. These examples both come from a similar construction, i.e., considering an abelian group with the semidirect product of $\mathbb{Z}/2\mathbb{Z}$. Thus we take the time to consider some general theory regarding these kinds of groups. Finally, we then give an example of a group where the invariant Ellis subgroup and the finitely satisfiable Ellis subgroups are not (abstractly) isomorphic. This is an abelian $\dfg$ example with the independence property.

\subsection{Generalized dihedral groups}

Fix a structure $M$ and let $A = (A;+,0,\dots)$ be a definable abelian group. Then $G = A \rtimes \{-1,1\}$ is also a definable group where $\{-1,1\}$ are two distinguished points and group multiplication on the Cartesian product $A \times \{1,-1\}$ is given by
\begin{equation*}
(a,i) \cdot (b,j) =
\left\{
\begin{array}{ll}
(a+b, j) & \text{if } i = 1, \\
(a - b, -j) & \text{if } i = -1.
\end{array}
\right.
\end{equation*}
We consider $A$ as a subgroup of $G$, slightly abusing notation, so that for $a\in A$ we write $a=(a,1)$. Likewise, we consider $S_A(\cU)$ as a subspace of $S_G(\cU)$, writing $p=(p,1)$ for $p\in S_A(\cU)$. Also, if $p \in S_{A}(\cU)$, we often write $p^{-1}$ as $-p$. Since all abelian groups are amenable and any extension of an amenable group is amenable, we have that $G = A \rtimes \{\pm 1\}$ is discretely amenable and thus also definably amenable.

    \begin{remark}
        In $G$, we have $(a,1)^{-1}=(-a,1)$ and $(a,-1)^{-1}=(a,-1)$.
    \end{remark}

    \begin{proposition}\label{prop:semi}
        Suppose $t\in S^{\inv}_G(\cU,M)$ and is idempotent [not necessarily strong right $f$-generic], and let $i_1,i_2,j_1,j_2\in \{-1,1\}$. Then
        \begin{enumerate}
            \item
            $t\in S_A(\cU)$.
            \item For $p,q\in S_A^{\inv}(\cU,M)$, we have $(p*q)^{-1}=(-p)*(-q)$.
            \item
            $-t$ is idempotent.
            \item For $p,q \in S_{A}^{\inv}(\cU,M)$, then
            $(i_1 \cdot p ,j_1)*(i_2 \cdot q,j_2)=((i_1 \cdot p)*(j_1i_2 \cdot q),j_1j_2)$.
            \item $\{(p,\pm 1): p \in \mathcal{F}_{r}^{A}\}$ is the unique minimal left ideal in $S_{G}^{\inv}(\cU,M)$.
            \item If $t*(-t)=t$, then $(-t) * t = -t$ and $\{(t,1),(t,-1)\}$ and $\{(-t,1),(-t,-1)\}$ are groups.
            \item If $t * (-t)= -t$, then $(-t) * t = t$ and $\{(t,1),(-t,-1)\}$ and $\{(t,-1),(-t,1)\}$ are groups.
        \end{enumerate}
    \end{proposition}
    \begin{proof} We prove the statements.
    \begin{enumerate}
        \item Notice that $A \cdot A \subseteq A$ and $(G \backslash A) \cdot (G \backslash A) \subseteq A$. Thus for every $p \in S_{G}^{\inv}(\cU,M)$, $p * p \in S_{A}^{\inv}(\cU,M)$. Hence, $t = t * t \in S_{A}^{\inv}(\cU,M)$.
        \item Let $b \models q|_{M}$ and $a \models p|_{Mb}$. Then
        \begin{equation*}
            (p * q)^{-1} = ((a,1) \cdot (b,1))^{-1} = (a + b, 1)^{-1} = (-a -b,1) = (-a,1) \cdot (-b,1) = (-p) * (-q).
        \end{equation*}
        \item Follows directly from statements $(1)$ $\&$ $(2)$.
        \item Straightforward calculation.
        \end{enumerate}
        $(5)-(7)$ directly from $(4)$.
    \end{proof}

\begin{example}\label{example:dfg} Consider $G = \mathbb{R} \rtimes \{-1,1\}$ as a group definable in the real closed field $M = (\mathbb{R};+,\times,0,1)$. Let $t_{+\infty} := \{x > a: a \in \cU\}$ and $t_{-\infty} := \{x < a: a \in \cU\}$. We let $t_{+}$ be the type corresponding to the cut right above $\mathbb{R}$ and $t_{-}$ to the cut right below $\mathbb{R}$. Then
\begin{enumerate}
    \item By Proposition \ref{prop:semi}(5), the unique minimal left ideal of $S_{G}^{\inv}(\cU,M)$ is precisely $\{(t_{\pm\infty},\pm 1)\}$.
    \item The Ellis subgroups of $\mathcal{F}_{r}$ are precisely $\{(t_{+ \infty}, \pm 1) \}$ and $\{(t_{-\infty}, \pm 1)\}$.
    \item The space $S_{G}^{\fs}(\cU,M)$ has two minimal left ideals, namely $\{(t_{+},1), (t_{-},-1)\}$ and $\{(t_{-},1),(t_{+},-1)\}$. Both minimal left ideals are also Ellis subgroups.
    \item Notice that $F_{M}(\{(t_{+\infty},\pm1)\} = \{(t_{+},\pm 1)\}$ and $F_{M}(\{(t_{-\infty},\pm1)\} = \{(t_{-},\pm 1)\}$. The images of the Ellis subgroups are not groups, e.g., $(t_{+},-1) * (t_{+},-1) = (t_{-},1)$.
    \item By Theorem \ref{theorem:map}, the map $K_{M}$ restricts to an anti-isomorphism from $\{(t_{+\infty},\pm 1)\}$ to $\{(t_{-},1),(t_{+},-1)\}$ and from $\{(t_{-\infty},\pm 1)\}$ to $\{(t_{+},1),(t_{-},-1)\}$. Since the Ellis subgroups are abelian, this anti-isomorphism is an isomorphism.
\end{enumerate}
\end{example}

\begin{example}\label{example:fsg} Consider $G = S^{1} \rtimes \{-1,1\}$ as a group definable in the real closed field $M = (\mathbb{R};+,\times,0,1)$. For any $a \in S^{1}(\mathbb{R})$ we let $p_{a}^{-}$ be the global finitely satisfiable type in $S^{1}(\mathbb{R})$ with standard part $a$ and for any $b \in S^{1}(\mathbb{R})$, $b < x < a \in p_{a}^{-}$ (in the sense of the natural circular ordering on $S^1$). Likewise, we let $p_{a}^{+}$ be the global finitely satisfiable type in $S^{1}(\mathbb{R})$ with standard part $a$ and for any $b \in S^{1}(\mathbb{R})$, $a < x < b \in p_{a}^{+}$ (in the sense of the natural circular ordering on $S^1$). Then
\begin{enumerate}
    \item By Proposition \ref{prop:semi}(5) and \cite[Example 4.2]{chernikov2022definable}, the unique minimal left ideal of $S_{G}^{\inv}(\cU,M)$ is precisely $\{(p_{a}^{\pm},\pm1): a \in S^{1}(\mathbb{R})\}$.
    \item The Ellis subgroups of $\mathcal{F}_{r}$ are $H_1:= \{((p_{a}^{+},\pm 1): a \in S^{1}(\mathbb{R})\}$ and $H_2 :=\{(p_{a}^{-}, \pm 1): a \in S^{1}(\mathbb{R})\}$.
    \item Since $G$ is $\fsg$, the unique minimal left ideal of $S_{G}^{\fs}(\cU,M)$ coincides with the unique minimal left ideal of $S_{G}^{\inv}(\cU,M)$. Thus, they have the same Ellis subgroups.
    \item Notice that
    \begin{equation*} K_{M}(H_1) = H_1^{-1} = \{(p_{a}^{-},1): a \in S^{1}(\mathbb{R})\} \cup \{(p_{a}^{+}, -1) : a \in S^{1}(\mathbb{R})\},
    \end{equation*}
    and,
    \begin{equation*}
        K_{M}(H_2) = H_2^{-1} = \{(p_{a}^{+},1): a \in S^{1}(\mathbb{R})\} \cup \{(p_{a}^{-}, -1) : a \in S^{1}(\mathbb{R})\}.
    \end{equation*}
    Clearly, neither of the above are Ellis subgroups (since the only Ellis subgroups of $S^{\fs}_{G}(\cU,M)$ are $H_1$ and $H_2$).
\end{enumerate}
\end{example}

\begin{remark} Both examples above demonstrate that Ellis subgroups are not preserved under model theoretic inversion. Example~\ref{example:fsg} also shows that even if a given type is bi-generic (i.e.\ both left and right $f$-generic), the left and the right orbits may not coincide.
\end{remark}

\begin{example}\label{example:product} Consider $G = (S^{1} \rtimes \{-1, 1\}) \times (\mathbb{R} \rtimes \{-1,1\})$ as a definable subgroup of the real closed field $M = (\mathbb{R};+,\times,0,1)$ with both the circular ordering and standard ordering in the respective coordinates. From the previous examples, it is clear that neither the retraction map $F_{M}$ nor the inverted retraction map $K_{M}$ maps Ellis subgroups to Ellis subgroups.
\end{example}

Finally, we make a quick remark regarding minimal two-sided ideals.

\begin{remark} Recall Example \ref{example:Heisenberg}. Note that since $[0,0,1]\in G(M)$ is not in the right stabiliser of $p|_M$, $[0,0,1]^{-1}=[0,0,-1]\in G(M)$ is not in the left stabilizer of $p^{-1}|_M$, and so it is not in the left stabilizer of $F_M(p^{-1})$. Hence, $F_M(p^{-1})$ is the retract of a $\dfg$ minimal idempotent in $S^\inv_G(\cU,M)$ which is itself not in the minimal ideal. This shows that in particular, in general, the $F_M$-retract of the minimal (two-sided) ideal of $S^\inv_G(\cU,M)$ is not contained in the minimal (two-sided) ideal of $S^\fs_G(\cU,M)$.
\end{remark}

\subsection{Invariant Ellis subgroups may not be isomorphic to finitely satisfiable Ellis subgroups} We provide an example of an abelian $\dfg$ group (outside of the NIP setting) such that the isomorphism type of the invariant Ellis subgroups is different from the isomorphism type of the finitely satisfiable Ellis subgroups.

Let $M$ be an atomless Boolean algebra in the language $\Lc _{Bool} = \{\cap,\cup, \leq,~^\complement, 0 , 1\}$. Then $G(x) := \text{`}x = x\text{'}$ is an abelian group with group multiplication defined by symmetric difference -- denoted $+$, group identity $0$, and group inverse $^\complement$. We recall that the theory admits quantifier elimination.

It is clearly abelian, and we will see promptly that it is $\dfg$ -- even extremely amenable. We will then show that the Ellis subgroup for invariant types is trivial, while the isomorphism type of the Ellis subgroups for finitely satisfiable types is not -- proving our claim.

\begin{fact}
    \label{fct:random_type}
    The set $\{x \cap a\neq 0\neq a\setminus x : a\in \cU\setminus \{0\}\}$ is consistent and extends to a unique global $\emptyset$-definable type $p_r$.
\end{fact}
\begin{proof}
    Straightforward, using quantifier elimination and the observation that any atomic formula is equivalent to a conjunction of formulas of the form $t=0$, where $t$ is a term using only $\cap$ and ${}^\complement$.
\end{proof}

\begin{corollary}
    $G$ is dfg.
\end{corollary}
\begin{proof}
    Clearly, $p_r$ is translation invariant and $\emptyset$-definable.
\end{proof}

\begin{fact} For any $a \in \cU$ such that $a \neq 0,1$, we have $p_r\vdash x\not\leq a\not\leq x$.
\end{fact}
\begin{proof}
     Indeed, if $p_r\vdash a\leq x$, then $p_r\vdash a\setminus x=0$, so $a=0$, and if $p_r\vdash x\leq a$, then $p_r\vdash x\setminus a=x\cap a^\complement=0$, so $a^\complement=0$, i.e.\ $a=1$.
\end{proof}

\begin{fact}
    \label{fct:ultrafilter_type}
    Let $M$ be an atomless Boolean algebra, and let $U\subseteq M$ be an ultrafilter. Then the set $\{0<x<c: c\in U\}$ extends uniquely to a type in $S_{x}(M)$.
\end{fact}
\begin{proof}
    Direct by quantifier elimination.
\end{proof}

\begin{proposition}
    \label{prop:ultrafilter_dichotomy}
    If $b\in \cU$ realizes a type as in Fact~\ref{fct:ultrafilter_type} and $p\in S_{x}^{\fs}(\cU,M)$, then $p\vdash x \cap b=0$ or $p\vdash x\geq b$.
\end{proposition}
\begin{proof}
    For any $c\in M$ either $c\in U$, in which case $b\leq c$, or $c^\complement\in U$, in which case $b\leq c^\complement$ and $c\cap b =0$.
\end{proof}

\begin{corollary}
    $p_r$ is not finitely satisfiable in any small model.
\end{corollary}

\begin{proposition}
    If $b\in\cU$ is arbitrary and $A\subseteq \cU$, then taking $M'$ to be the Boolean subalgebra generated by $Ab$, for any $a\models p_r|M'$ we have $a+b\models p_r|M'$.
\end{proposition}
\begin{proof}
    If $c\in M'$ is arbitrary, then:
    \begin{itemize}
        \item
        if $c\not \leq b$, then $(a+b)\cap c\geq (a+b)\cap(c\setminus b)=a\cap (c\setminus b)\neq 0$, and
        \item
        if $c\leq b$, then $(a+b)\cap c=(a\cap c)+(b\cap c)=(a\cap c)+c=c\setminus a\neq 0$.
    \end{itemize}

    Likewise,
    \begin{itemize}
        \item
        if $c\not\leq b$, then
        $c\setminus (a+b)\geq (c\setminus b)\setminus(a+b)=(c\setminus b)\setminus a\neq 0$, and
        \item
        if $c\leq b$, then $c\setminus(a+b)=c\cap a\neq 0$.\qedhere
    \end{itemize}
\end{proof}

\begin{corollary}
    For any $q \in S_{G}^{\inv}(\cU,M)$, $p_{r} * q = p_{r}$.
\end{corollary}

\begin{corollary}\label{cor:inv_ba}
    The Ellis subgroup of $S_{G}^{\inv}(\cU,M)$ is trivial.
\end{corollary}
\begin{proof}
    We have that $p_r*S_{G}^{\inv}(\cU,M)=\{p_r\}$ so by Fact~\ref{fact:right}(11) we see that $\{p_r\}$ is an Ellis group.
\end{proof}

\begin{proposition}\label{prop:fs_ba}
    The Ellis subgroup of $S_{G}^{\fs}(\cU,M)$ is nontrivial.
\end{proposition}
\begin{proof}
    Let $U$ be an ultrafilter on $M$, let $p_U$ be an arbitrary global coheir of a type as in Fact~\ref{fct:ultrafilter_type}, and let $b\models p_U|M$.

    Notice that by Proposition~\ref{prop:ultrafilter_dichotomy}, given any $p,q\in S_{G}^{\fs}(\cU,M)$, we have that $p*q\vdash x>b$ if and only if for exactly one of $p\vdash x>b,q\vdash x>b$ holds.

    It follows that any idempotent $u\in S_{G}^{\fs}(\cU,M)$ must satisfy $u\vdash x\cap b=0$, and yet $u*p_U*u\vdash x>b$, whence $u\neq u*p_U*u$. Thus if we take $u$ to be a minimal idempotent in $S^\fs_G(\cU,M)$, then $u$ and $u*p_U*u$ are distinct elements of the Ellis group $u*S^\fs_G(\cU,M)*u$.
\end{proof}

\begin{corollary}\label{cor:different} The isomorphism types of the Ellis subgroups of $S_{G}^{\inv}(\cU,M)$ and $S_{G}^{\fs}(\cU,M)$ are different.
\end{corollary}

\begin{proof} Direct by \cref{cor:inv_ba} and \cref{prop:fs_ba}.
\end{proof}

\section{ Analysis of semidirect products}

In this final section, we apply our intuition obtained from the previous sections to study the Ellis theory of semidirect products of $\fsg$ and $\dfg$ groups. One motivation to study groups of this form is the open question asking whether in NIP, or at least distal theories, every definably amenable group is the extension of a dfg group by an fsg group -- of which the semidirect product is a special case where the extension is split. (Cf.\ e.g.\ \cite[Proposition 4.6]{CP12}, \cite[Question 1.19]{pillay2016minimal} and \cite{johnson2025abelian} for context and some partial results.)
Our main theorem (Theorem \ref{thm:main_semidirect}) describes invariant Ellis subgroups as the Newelski product of invariant Ellis subgroups of each component. Similarly, we can also describe the finitely satisfiable Ellis subgroups by inverting, retracting and swapping the order of multiplication. This is consistent with our intuition from the previous section of the paper; on the $\fsg$ portion, we \emph{do nothing} and on the $\dfg$ portion, we \emph{invert and retract}.

Throughout this entire section, we assume NIP. As usual, we have a fixed first order theory $T$ where $\cU$ is a monster model of $T$ and $M$ is a small elementary submodel of $\cU$.

Some of the lemmas we prove are stated in a more general form than required for just the semidirect product -- with the hope that they can perhaps be adapted to use in a more general case of non-split extensions of definably amenable groups.

\begin{remark}
    Let $G,H,K$ be abstract groups. Recall that:
    \begin{enumerate}
        \item
        Given a left action $\alpha$ of $K$ on $H$ by automorphisms, the \emph{external semidirect product} $H\rtimes_\alpha K$ is a group with the underlying set $H\times K$ and multiplication given by
        \[
            (h_1,k_1)\cdot (h_2,k_2)=(h_1\cdot\alpha(k_1,h_2),k_1\cdot k_2)
        \]
        \item
        We say that $G$ is the \emph{internal semidirect product of $H$ and $K$}, written (abusing the notation) $G=H\rtimes K$, if $H\unlhd G$, $K\leq G$, $H\cap K=\{1\}$ and $H\cdot K=G$.
        \item
        Equivalently, $G$ is the internal semidirect product of $H$ and $K$ if $H\unlhd G$, $K\leq G$ and for the action $\alpha\colon K\times H\to H$ given by $\alpha(k,h)=h^{k^{-1}}=khk^{-1}$ we have an isomorphism $H\rtimes_\alpha K\to G$ given by the formula $(h,k)\mapsto h\cdot k$.
        \item
        Equivalently, $G$ is the semidirect product of $H$ and $K$ if we have a short exact sequence $H\to G\to K$ of groups which is split in the sense that the right arrow has a section $K\to G$ which is a homomorphism.
        \item
        Conversely, given an isomorphism $\phi\colon H\rtimes_\alpha K\to G$, $G$ is the internal semidirect product of $\phi(H)$ and $\phi(K)$, and $\alpha(k,h)=\phi^{-1}(\phi(h)^{\phi(k)^{-1}})$.
        \item
        In the definition of the internal semidirect product, we can also write $G=K\cdot H$ instead of $G=H\cdot K$ (simply by noting that $(K\cdot H)^{-1}=H^{-1}\cdot K^{-1}=H\cdot K$). This corresponds to the anti-isomorphism $H\rtimes_\alpha K\cong K\ltimes_\beta H$ given by the formula $(h,k)\mapsto (k^{-1},h^{-1})$, where $K\ltimes_\beta H$ is $K\times H$ with multiplication given by $(k_1,h_1)\cdot (k_2,h_2)=(k_1\cdot k_2, \beta(h_1,k_2)\cdot h_2)$, where $\beta$ is the right action of $K$ on $H$ given by $\beta(h,k)=\alpha(k^{-1},h^{-1})^{-1}$ (internally, $(h,k)\mapsto k^{-1}hk$).
    \end{enumerate}
\end{remark}

\begin{remark}
\label{rem:def_semidirect_product}
Let $G,H,K$ be $\emptyset$-definable groups with $H,K\leq G$. It is not hard to see that the following are equivalent:
    \begin{itemize}
        \item
        $G(M)$ is the internal semidirect product of $H(M)$ and $K(M)$,
        \item
        we have a definable split short exact sequence $H\to G\to K$ (i.e.\ all arrows and the splitting homomorphisms are $\emptyset$-definable) where $H\to G$ and the splitting homomorphism $K\to G$ are inclusions,
        \item
        $G$ is naturally (i.e.\ via the multiplication map) $\emptyset$-definably isomorphic to $H\rtimes_\alpha K$, where $\alpha$ is the conjugation action.
    \end{itemize}
    In this case, we say that $G$ \emph{is the definable semidirect product of $H$ and $K$} or \emph{is the definable semidirect product $H\rtimes K$}. Slightly abusing the notation, we will sometimes identify $G$ with $H\rtimes_\alpha K$, and identify $H$ and $K$ with the corresponding subgroups of $H\rtimes_\alpha K$.
\end{remark}

    \begin{remark}
        \label{rem:conn_component_normal}
        Suppose that $G(x)$ is an $\emptyset$-definable group and $H(x)$ is an $\emptyset$-definable normal subgroup. Then for each $g\in G(\cU)$,
        $H^{00}(\cU)^g$ is a type-definable subgroup of $H(\cU)$ of bounded index in $H(\cU)$, hence contained in $H^{00}(\cU)$, so $H^{00}(\cU)$ is a normal subgroup of $G(\cU)$.
    \end{remark}

    \begin{fact}
        \label{fct:semidirect_component}
        If $G$ is the definable semidirect product of $H$ and $K$, then $G^{00}(\cU)$ is the (internal) semidirect product of $H^{00}(\cU)$ and $K^{00}(\cU)$ (i.e.\ $G^{00}(\cU)=H^{00}(\cU)\rtimes K^{00}(\cU)$, equivalently by Remark~\ref{rem:conn_component_normal}, $G^{00}(\cU)=H^{00}(\cU)\cdot K^{00}(\cU)$) and $G(\cU)/G^{00}(\cU)=(H(\cU)/H^{00}(\cU))\rtimes (K(\cU)/K^{00}(\cU))$.
    \end{fact}
    \begin{proof}
        The first part is an immediate consequence of Remark~\ref{rem:conn_component_normal} and \cite[Corollary 4.11]{GJK23}.
        For the second part, notice that $H(\cU)\cap (H^{00}(\cU)\cdot K^{00}(\cU))=H^{00}(\cU)$, so by the first part, $H(\cU)/G^{00}(\cU)=H(\cU)/H^{00}(\cU)$, and likewise for $K$.
    \end{proof}

    \begin{proposition}
        \label{prop:enough_for_H00_cosets}
        Fix a $\emptyset$-definable short exact sequence of definable groups $H\to G\overset{\pi_K}{\to} K$. Identify $H$ with its image in $G$. Suppose either of the following holds:
        \begin{itemize}
            \item
            $H^{00}(\cU) =G^{00}(\cU) \cap H(\cU)$,
            \item
            we have an $M$-definable section $K\to G$ (not necessarily homomorphic).
        \end{itemize}
        Then whenever $g_1,g_2\in G(\cU)$ are such that $\pi_K(g_1)=\pi_K(g_2)$ and $g_1\equiv_{M} g_2$ (recalling that $M$ is a fixed small model), we have $g_2g_1^{-1}\in H^{00}(\cU)$.
    \end{proposition}
    \begin{proof}
        Let $k\coloneqq \pi_K(g_1)=\pi_K(g_2)$.

        Note that $\pi_K(g_2g_1^{-1})=e_K$, so by exactness of our sequence we have that $g_2g_1^{-1}\in H(\cU)$. Moreover, since $g_1\equiv_{M} g_2$, it follows that $g_2g_1^{-1}\in G^{00}(\cU)$. Thus if $G^{00}(\cU) \cap H(\cU)=H^{00}(\cU)$, the conclusion is true.

        On the other hand, if we have a definable section, write $g$ for the image of $k$ under the definable section map. Then, as $k=\pi_K(g_1)=\pi_K(g_2)$ and $g\in\operatorname{dcl}(Mk)$, we have that $(g_1,g)\equiv_{M} (g_2,g)$, hence $g g_1^{-1}\equiv_{M} gg_2^{-1}$. As $gg_2^{-1}, g g_1^{-1}\in H(\cU)$, it follows that $( gg_2^{-1})^{-1}(g g_1^{-1})=g_2g_1^{-1}\in H^{00}(\cU)$.
    \end{proof}

    \begin{remark}
        Note that the conclusion of Proposition~\ref{prop:enough_for_H00_cosets} is not true in general, even when the groups are abelian and stable. For example, working in the eq-expansion of $\mathrm{ACF}_0$, if we take for $G$ the multiplicative group of the field, $H=\{1,-1\}$ and $M$ is arbitrary, then for any element $a\notin M$ we have $a\equiv_M -a$ and $\pi_K(a)=\{a,-a\}=\pi_K(-a)$, but $(-a)\cdot a^{-1}=-1\notin H^{00}(\cU)=\{1\}$. Both variant hypotheses of Proposition~\ref{prop:enough_for_H00_cosets} (and hence also the conclusion) hold for definable semidirect products -- the first one by Fact~\ref{fct:semidirect_component}, and the second one essentially by definition.
    \end{remark}

    \begin{proposition}
        \label{prop:fgen_in_ext}
        Suppose we have a definable short exact sequence of groups $H\to G\overset{\pi_K}\to K$ satisfying the conclusion of Proposition~\ref{prop:enough_for_H00_cosets}, and suppose that $H,K$ are definably amenable with $p_H\in S^\inv_H(\cU,M)$, $p_K\in S^\inv_K(\cU,M)$ strongly right $f$-generic.
        Let $q_0\in S_G(\cU)$
        be such that $\pi_K(q_0)=p_K$. Then, identifying $H$ with its image in $G$:
        \begin{itemize}
            \item
            if $q_1,q_2\in S_G(\cU)$ are such that $q_1|_{M}=q_2|_{M}$ and $\pi_K(q_1)=\pi_K(q_2)$, then $p_H*q_1=p_H*q_2$,
            \item
            if $g\equiv_M g'$, then we have $(p_H*q_0)\cdot g=(p_H*q_0)\cdot g'$,
            \item
            $p_H*q_0\in S^\inv_G(\cU,M)$ and is strongly right $f$-generic
        \end{itemize}
    \end{proposition}
    \begin{proof}
        Fix an arbitrary $g\in G(\cU)$ and small models $N_1,N_2$ such that $M \preceq^{+} N_1 \preceq^{+} N_2$ with $g\in N_1$.

        For the first bullet, fix $k\models \pi_K(q_1)|_{N_1}$ in $N_2$, some $g_1,g_2\in G(N_2)$ such that $g_j\models q_j|_{N_1}$ and $\pi_K(g_j)=k$ for $j=1,2$, as well as $h\models p_H|_{N_2}$ in $H(\cU)$. Then $hg_j\models (p_H*q_j)|_{N_1}$, so we only need to show that $hg_1\equiv_{N_1} hg_2$. Setting $h'\coloneqq g_2g_1^{-1}$ -- so that $hg_2=hh'g_1$ -- by the conclusion of Proposition~\ref{prop:enough_for_H00_cosets}, we have $h'\in H^{00}(N_1)$. But then -- since $p_H$ is right $f$-generic -- $hh'\models (p_H\cdot h')|_{N_2}=p_H|_{N_2}$, so $hg_2=hh'g_1\models (p_H*q_1)|N_1=\tp(hg_1/N_1)$.

        For the second bullet, note that if $g\equiv_M g'$, then $q_0 \cdot g|_M = q_0 \cdot g'|_M$ and $\pi_K(g)K^{00}=\pi_K(g')K^{00}$, so $\pi_K(q_0\cdot g)=p_K\cdot \pi_K(g)=p_K\cdot \pi_K(g')=\pi_K(q_0\cdot g')$. By the first bullet we conclude that
        \[
            (p_H*q_0)\cdot g=p_H*(q_0\cdot g)=p_H*(q_0\cdot g')=(p_H*q_0)\cdot g'.
        \]

        For the third bullet, note that there is a type $q_1\in S_G^\inv(\cU,M)$ such that $q_1|_{M} = q_0|_M$ and $\pi_K(q_0)=\pi_K(q_1)$. Indeed, $q_0|_M\cup \{\varphi(\pi_K(x)): \varphi(x)\in p_K\}$ is a partial $M$-invariant type and any $q_1\in S_G^\inv(\cU,M)$ extending it suffices. Then by the first bullet $p_H*q_0=p_H*q_1\in S_G^\inv(\cU,M)$. Fixing a small model $M'$ such that $M \preceq^{+} M'$, by the second bullet, we have that $(p_H*q_0)\cdot G(\cU)=(p_H*q_0)\cdot G(M')$, so the $G(\cU)$-orbit of $p_H*q_0$ is small, so by Fact~\ref{fct:fgen_small_orbit} it is indeed right $f$-generic.
    \end{proof}

    \begin{corollary}\label{corollary:constuct}
        Suppose we have a definable short exact sequence $H\to G\overset{\pi_K}\to K$ of definably amenable groups satisfying the one of the variant hypotheses of Proposition~\ref{prop:enough_for_H00_cosets}. Fix Ellis subgroups $\widehat{E_H}$ and $E_K$ in $S_H^\inv(\cU,M)$ and $S_K^\inv(\cU,M)$ respectively. Let $\widehat E_K\subseteq S_G(\cU)$ be such that $\pi_K[\widehat{E_K}]=E_K$. Then $\widehat{E_H}*\widehat{E_K}$
        is an Ellis subgroup in $S_G^\inv(\cU,M)$.
    \end{corollary}
    \begin{proof}
        Note that by Proposition~\ref{prop:fgen_in_ext}, we know that $\widehat{E_H}*\widehat{E_K}$ consists of $f$-generic types.
        By Theorem~\ref{theorem:main}, we need to show that it is a single $G(\cU)$-orbit.

        Fix $p_H\in \widehat{E_H}$ and $q_1\in \widehat{E_K}$, take any $g\in G(\cU)$. Since $\pi_K(q_1\cdot g)=\pi_K(q_1)\cdot \pi_K(g)\in E_K$ (because, e.g.\ by Theorem~\ref{theorem:main}, $E_K=\pi_K(q_1)\cdot K(\cU)$), we have that by hypothesis, there is some $q_2\in \widehat{E_K}$ such that $\pi_K(q_2)=\pi_K(q_1\cdot g)$.

        Note that if $g_1,g_1'\models q_1|_{Mg}$ and $g_2,g_2'\models q_2|_{Mg}$ satisfy $\pi_K(g_1g)=\pi_K(g_2), \pi_K(g_1'g)=\pi_K(g_2')$, then $(g_2')^{-1}g_1'g\in g_2^{-1}g_1gH^{00}(\cU)$.

        Indeed, if $H^{00}(\cU)=G^{00}(\cU)\cap H(\cU)$, then this is true because $g_2G^{00}(\cU)=g_2'G^{00}(\cU)$ and $g_1G^{00}(\cU)=g_1'G^{00}(\cU)$, and if we have a definable section $s\colon K\to H$, then this is true because, setting $k\coloneqq \pi_K(g_2),k'=\pi_K(g_2')$, we have that $(g_2,k)\equiv_M(g_2',k')$, so $s(k)^{-1}g_2 \equiv_M s(k')^{-1}g_2'$, hence $s(k)^{-1}g_2\in s(k')^{-1}g_2'H^{00}(\cU)$, and similarly, since $k=\pi_K(g_1g)$ and $k'=\pi_K(g_1'g)$:
        \[s(k)^{-1}g_1g\in s(k')^{-1}g_1'gs(k')^{-1}H^{00}(\cU),\]
        whence (e.g.\ using Remark~\ref{rem:conn_component_normal}):
        \begin{align*}
            g_2^{-1}g_1g&=g_2^{-1}s(k)s(k)^{-1}g_1=(s(k)^{-1}g_2)^{-1}s(k)^{-1}g_1g\\
            &\in (s(k')^{-1}g_2'H^{00}(\cU))^{-1}s(k')^{-1}g_1'gH^{00}(\cU)\\
            &\in (g_2')^{-1}s(k')s(k')^{-1}g_1'g_2H^{00}(\cU)\\
            &\in (g_2')^{-1}g_1'g_2H^{00}(\cU).
        \end{align*}

        It follows that there is a $h_0\in H(\cU)$ such that for all $g_1\models q_1|_{Mg}$ and $g_2\models q_2|_{Mg}$ such that $\pi_K(g_1g)=\pi_K(g_2)$ we have $g_1g\in h_0H^{00}(\cU)g_2$. Since $p_H$ is right $f$-generic, a straightforward argument shows that
        \[
            (p_H*q_1)\cdot g=p_H*(q_1\cdot g)=(p_H\cdot h_0)*q_2\in \widehat{E_H}*\widehat{E_K}.
        \]

        Conversely, for any $p_H'\in \widehat{E_H}$ and $q'\in \widehat{E_K}$ we want to show that $p_H'*q'$ is in the $G(\cU)$-orbit of $p_H*q_1$. Let $g_0\in G(\cU)$ be such that $\pi_K(q_1)\pi_K(g_0)=\pi_K(q')$ (this exists because $\pi_K(q_1),\pi_K(q_2)\in E_K$, which is a single $K(\cU)$-orbit). Then arguing as above we see that $(p_H*q_1)\cdot g_0=p_H''*q'$ for some $p_H''\in \widehat{E_H}$. Letting $h'$ be such that $p_H''\cdot h'=p_H'$, suppose $g'\models q'|_{Mh'}$ and let $h''=g'h'(g')^{-1}$. Then $\tp(h''/M)$, and hence $H^{00}(\cU)$ does not depend on the choice of $g'$, from which it follows that
        \[
            p_H'*q'=(p_H''\cdot h')*q'=p_H''*(q'\cdot h'')=(p_H*q)\cdot(g_0h'').\qedhere
        \]
    \end{proof}

    \begin{corollary}
        \label{cor:ellis_inv_semidirect}
        Suppose $G=H\rtimes K$ with $H, K$ definably amenable. Let $E_K, \widehat{E_H}$ be Ellis subgroups in $S_K^\inv(\cU,M)$ and $S_H^\inv(\cU,M)$ respectively. Then $\widehat{E_H}*E_K$ is an Ellis subgroup in $S_G^\inv(\cU,M)$. Furthermore, every element of this Ellis group is expressed uniquely as $p*q$ with $p\in \widehat{E_H}$ and $q\in E_K$.
    \end{corollary}
    \begin{proof}
        The inclusion of $K$ into $G$ is a definable section as in the hypothesis of Proposition~\ref{prop:enough_for_H00_cosets}. To see the uniqueness, note that $\hat \pi(p*q)=\hat\pi(p)\cdot \hat\pi(q)$ and $(p,q)\mapsto\hat\pi(p)\cdot \hat\pi(q)$ is injective on $\widehat{E_H}\times E_K$ by Theorem~\ref{theorem:main} and Fact~\ref{fct:semidirect_component}.
    \end{proof}

    \begin{lemma}
        \label{lem:fsg_and_left_ideal}
        Suppose $G$ is a definable group which is the definable semidirect product $H\rtimes K$, where $K$ is fsg. Let $L$ be a left ideal in $S_H^\fs(\cU,M)$, and let $E_K$ be an Ellis subgroup in $S_K^\fs(\cU,M)$. Then $E_K*L=E_K*S_G^\fs(\cU,M)*L$.
        In particular, $E_K*L$ is a subsemigroup of $S_G^\fs(\cU,M)$.
    \end{lemma}
    \begin{proof}
        Fix arbitrary $p\in L, q\in E_K, r\in S_G^\fs(\cU,M)$ and $M \preceq^{+} N$, and consider $(k,g,h)\models q\otimes r\otimes p|N$, so that $kgh\models q*r*p|_N$

        By hypothesis, we can write $g$ as $k'h'$ -- so that in particular, $\tp(h'/N)$ is finitely satisfiable in $M$ and $g$ is interdefinable with $(k',h')$. Letting $q'\coloneqq q\cdot k'$ and $p'=\tp(h'/N)|_\cU$, we have that $kk'\models q'|Nk'h'h$ and $h'h\models p'*p$. Since $q$ is fsg, it is $f$-generic, so $q'\in E_K$, and since $L$ is a left ideal, $p'*p\in L$. It is not hard to see that $q'$ and $p'$ do not depend on the choices we made (the $K^{00}(\cU)$-coset of $k'$ and $p'$ both depend only on $r$). It follows that $q*r*p=q'*p'*p$ is in $E_K*L$.
    \end{proof}

    \begin{corollary}
        \label{cor:ellis_fs_semidirect}
        Suppose $G$ is the definable semidirect product $H\rtimes K$, where $K$ is fsg and $H$ is dfg or, more generally, $H$ is definably amenable with a unique minimal right ideal in $S_H^\fs(\cU,M)$ (equivalently, by the dual of Fact~\ref{fact:equiv}, such that some (equivalently, every) Ellis group in $S_H^\fs(\cU,M)$ is a left ideal -- in particular, if there is a type which satisfies any of the conditions (1)-(7) in Theorem~\ref{thm:main_dfg_alternative}).

        Then if $E_K,E_H$ are Ellis subgroups in $S^\fs_K(\cU,M),S^\fs_H(\cU,M)$ respectively, then $E_K*E_H$ is an Ellis subgroup in $S_G^\fs(\cU,M)$. Furthermore, every element of this Ellis group is expressed uniquely as $q*p$ with $q\in E_K$ and $p\in {E_H}$.
    \end{corollary}
    \begin{proof}
        Note that if $H$ is dfg, then any dfg type satisfies \eqref{it:pG_closed} in Theorem~\ref{thm:main_dfg_alternative}, and if such a type exists, then \eqref{it:image_min_ideal} in that Theorem yields an Ellis group in $S_H^\fs(\cU,M)$ which is a left ideal (which by Fact~\ref{fact:right} implies that there is a unique minimal right ideal and that all Ellis groups are left ideals). Thus in all cases we have that $E_H$ is a left ideal in $S_H^\fs(\cU,M)$.

        Let $R$ be a minimal right ideal of $S_G^\fs(\cU,M)$ contained in $E_K*S_G^\fs(\cU,M)$, and let $L$ be a minimal left ideal contained in $S_G^\fs(\cU,M)*E_H$. Note that then, by Lemma~\ref{lem:fsg_and_left_ideal}, we have that
        \[
            R*L\subseteq E_K*S_G^\fs(\cU,M)*S_G^\fs(\cU,M)*E_H=E_K*E_H.
        \]

        By Fact~\ref{fact:right}, $R*L$ is an Ellis group in $S_G^\fs(\cU,M)$, and Fact~\ref{fact:NC} along with Fact~\ref{fct:semidirect_component} imply that $\hat \pi$ induces a bijection $R*L\to G(\cU)/G^{00}(\cU)$ and bijections $E_K\to K(\cU)/K^{00}(\cU)$ and $E_H\to H(\cU)/H^{00}(\cU)$. Since -- also by Fact~\ref{fct:semidirect_component} -- $G(\cU)/G^{00}(\cU)=H(\cU)/H^{00}(\cU)\rtimes K(\cU)/K^{00}(\cU)$, we conclude that $\hat \pi$ restricted to $E_K*E_H$ is also bijective with $G(\cU)/G^{00}(\cU)$, so $E_K*E_H=R*L$.

        To see the uniqueness, note that $\hat \pi(q*p)=\hat\pi(q)\cdot \hat\pi(p)$ and $(q,p)\mapsto\hat\pi(q)\cdot \hat\pi(p)$ is injective on $E_K\times E_H$ by Fact~\ref{fact:NC} and Fact~\ref{fct:semidirect_component}.
    \end{proof}

    \begin{remark}
        In Corollaries~\ref{cor:ellis_inv_semidirect} and \ref{cor:ellis_fs_semidirect}, one might be tempted to say that $\widehat{E_H}* {E_K}$ -- respectively $E_K*E_H$ -- is the internal semidirect product of $\widehat{E_H}$ and ${E_K}$ (resp.\ $E_H$ and $E_K$) but this is not true, since in general, $\widehat{E_H},{E_K}$ (resp.\ $E_H,E_K$) are not contained in $\widehat{E_H}*{E_K}$ (resp.\ $E_K*E_H$).

        Of course, in both cases, the Ellis group is isomorphic to the external semidirect product of the corresponding groups, and is the internal semidirect product of the appropriate subgroups, but this essentially follows immediately from Theorem~\ref{theorem:main} and Facts~\ref{fact:NC} and \ref{fct:semidirect_component}.
    \end{remark}

The following theorem summarizes the observations of this section for semidirect products of dfg and fsg groups, and gives somewhat explicit formulas for isomorphisms between Ellis groups of invariant and finitely satisfiable types in this case.
\begin{theorem}
\label{thm:main_semidirect}
Suppose that $T$ is NIP and $G$ is the definable semidirect product $H\rtimes K$, where $K$ is fsg and $H$ is dfg (see Remark~\ref{rem:def_semidirect_product}). Let $E_K$ be any Ellis group in $S_K^\fs(\cU,M)$ and let $\widehat{E_H}\subseteq S_H^\inv(\cU,M)$ be an Ellis group of dfg types. Set $E_H\coloneqq K_M(\widehat{E_H})=F_M(\widehat{E_H})^{-1}$.

Then
\begin{enumerate}
    \item
    $E_H$ is an Ellis group in $S_H^\fs(\cU,M)$,
    \item
    $\widehat{E_H}*E_K$ is an Ellis group in $S_G^\inv(\cU,M)$, and every element is expressed uniquely as $p*q$ with $p\in \widehat{E_H}$ and $q\in E_K$,
    \item
    $E_K*E_H$ is an Ellis group in $S_G^\fs(\cU,M)$, and every element is expressed uniquely as $q*p$ with $q\in E_K$ and $p\in E_H$,
    \item
    $p*q\mapsto q*K_M(p)$ defines a bijection $\widehat{E_H}*E_K\to E_K*E_H$,
    \item
    $\nu_M\colon p*q\mapsto q*\iota(\alpha(q^{-1}\otimes K_M(p))*u_H)$, defines an isomorphism $\widehat{E_H}*E_K\to E_K*E_H$, where $\iota\colon E_H\to E_H$ is the group inversion, $\alpha$ is the left action $K\times H\to H$ given by $\alpha(k,h)=khk^{-1}$, and $u_H\in E_H$ is the idempotent,
    \item
    $\nu'_M\colon p*q\mapsto u_K*q^{-1}*K_M(p)$, where $u_K\in E_K$ is the idempotent, defines an anti-isomorphism $\widehat{E_H}*E_K\to E_K*E_H$.
\end{enumerate}
\end{theorem}
\begin{proof}
    (1) follows from Theorem~\ref{theorem:map}. (2) follows from Corollary~\ref{cor:ellis_inv_semidirect}, (3) from Corollary~\ref{cor:ellis_fs_semidirect}. In (4), surjectivity is trivial, and injectivity follows from Theorem~\ref{theorem:main}, the fact that $\hat\pi\circ F_M=\hat\pi$ and Fact~\ref{fct:semidirect_component}.

    For (5), note that the image of $\nu_M$ is clearly contained in $E_K*E_H$, so by (2), (3) and Remark~\ref{rem:isomorphism_shortcut}, it is enough to show that $\hat \pi = \hat\pi\circ \nu_M$ on $\widehat{E_H}*E_K$.

    Set $g_1=\hat \pi(q),g_2=\hat\pi(p)$, so that $\hat\pi(p*q)=g_2g_1$. Then:
    \[
        \hat \pi(q*\iota(\alpha(q^{-1}\otimes K_M(p))*\hat u_H))=g_1\cdot ((g_1^{-1}g_2^{-1}g_1)\cdot 1)^{-1}=g_1\cdot(g_1^{-1}g_2g_1)=g_2g_1.
    \]

    The proof of (6) is analogous to (5) -- we only need to observe that on $\widehat{E_H}*E_K$, the composition $\hat \pi\circ \nu_M'$ is the (group-theoretic) inverse of $\hat\pi$.
\end{proof}

\begin{remark}
    In Theorem~\ref{thm:main_semidirect}, instead of assuming that $H$ is dfg and $\widehat{E_H}$ consists of dfg types, we can more generally assume that $H$ is definably amenable and $\widehat{E_H}$ is a closed Ellis group, or -- perhaps even more generally -- that there is a type $p\in S_H^\inv(\cU,M)$ satisfying the conditions (2)-(7) in Theorem~\ref{thm:main_dfg_alternative} and $\widehat{E_H}=p\cdot H(\cU)$ for such $p$. The proof of this more general case is essentially the same, only we apply Theorem~\ref{thm:main_dfg_alternative} instead of Theorem~\ref{theorem:map}.
\end{remark}

\begin{remark}
    Example~\ref{example:dfg} shows that the order of factors in Theorem~\ref{thm:main_semidirect} and Corollaries~\ref{cor:ellis_inv_semidirect} and \ref{cor:ellis_fs_semidirect} is crucial (and in particular, the roles of $H$ and $K$ are not at all symmetric in the conclusion). Indeed, note that in that example, $H,K$ are abelian, $H$ is dfg and $K$ is finite (hence fsg and dfg), but for $E_K=\{1,-1\}$, $\widehat{E_H}=\{t_{+\infty}\}$ and $E_H=\{t_-\}$ we have that $E_K*\widehat{E_H}$ and $E_H*E_K$ are not Ellis groups even under these rather strong additional assumptions. Similarly, Example~\ref{example:fsg} (in which (1) and (3) of Theorem~\ref{thm:main_semidirect} fail) shows that in Corollary~\ref{cor:ellis_fs_semidirect}, we cannot replace the hypothesis about $H$ with the hypothesis that it is also fsg and $K$ is finite (hence both fsg and dfg), and in particular, assuming only that $H$ and $K$ are definably amenable is not sufficient.
\end{remark}

\begin{question}
    \label{qu:can_we_omit_uh}
    Can $u_H$ be omitted in Theorem~\ref{thm:main_semidirect}(5)? In other words, is it true under the assumptions of the theorem that $\alpha(q^{-1}\otimes K_M(p))\in E_H$? (Equivalently, that $\alpha(q^{-1}\otimes K_M(p))*u_H=\alpha(q^{-1}\otimes K_M(p))$.) If not, are there some natural sufficient conditions for that to hold?
\end{question}

\begin{remark}
    The formula for $\nu_M$ can also be written in various other ways, for example, as
    \[
        \nu_M\colon p*q\mapsto q*\iota(\beta(K_M(p)\otimes q)*u_H),
    \]
    where $\beta$ is the right action $(h,k)\mapsto k^{-1}hk$ (to see that this formula is correct just notice that the composition with $\hat \pi$ is the same). We can also ask a variant of Question~\ref{qu:can_we_omit_uh} here -- can we omit $u_H$ in this formula?
\end{remark}

\begin{remark}
    It is not hard to check that if $K$ is abelian and we omit $u_K$ in (6), then we get an isomorphism $\widehat E_H *E_K\to E_K^{-1}*E_H$, and in this case $E_K^{-1}$ is an Ellis group. In contrast, if we take $G=K=S^1\rtimes \{1,-1\}$ and trivial $H$, Example~\ref{example:fsg} shows that in (6), $u_K$ cannot be omitted even when $H$ is trivial and $K$ is very tame, since generally $E_K^{-1}*E_H$ may not be an Ellis group. On the other hand, it is not hard to see that $u_K$ can be replaced by any minimal idempotent in $S_K^\fs(\cU,M)$.
\end{remark}

\printbibliography

@article{chernikov2022definable,
  title={Definable convolution and idempotent Keisler measures},
  author={Chernikov, Artem and Gannon, Kyle},
  journal={Israel Journal of Mathematics},
  volume={248},
  number={1},
  pages={271--314},
  year={2022},
  publisher={Springer},
  doi={10.1007/s11856-022-2298-2},
}

@article{chernikov2023definable,
  title={Definable convolution and idempotent Keisler measures, II},
  author={Chernikov, Artem and Gannon, Kyle},
  journal={Model Theory},
  volume={2},
  number={2},
  pages={185--232},
  year={2023},
  publisher={Mathematical Sciences Publishers},
  doi={10.2140/mt.2023.2.185},
}

@book {pier_amenable,
  title={Amenable Locally Compact Groups},
  author={Pier, Jean-Paul},
  year={1984},
  publisher={Wiley},
  isbn={0471893900},
}

@article{ellis1957locally,
  title={Locally compact transformation groups},
  author={Ellis, Robert},
  journal={Duke Mathematical Journal},
  volume={24},
  number={2},
  pages={119--125},
  year={1957},
  doi={10.1215/S0012-7094-57-02417-1},
  publisher={Duke University Press}
}

@article{johnson2025abelian,
  title={Abelian groups definable in p-adically closed fields},
  author={Johnson, Will and Yao, Ningyuan},
  journal={The Journal of Symbolic Logic},
  volume={90},
  number={1},
  pages={460--481},
  year={2025},
  doi={10.1017/jsl.2023.52},
  publisher={Cambridge University Press}
}

@article{CS,
  title={Definably amenable {NIP} groups},
  author={Chernikov, Artem and Simon, Pierre},
  journal={Journal of the American Mathematical Society},
  volume={31},
  number={3},
  pages={609--641},
  doi={10.1090/jams/896},
  year={2018},
}

@article{N1,
  title={Topological dynamics of definable group actions},
  author={Newelski, Ludomir},
  journal={The Journal of Symbolic Logic},
  doi={10.2178/jsl/1231082302},
  volume={74},
  number={1},
  pages={50--72},
  year={2009},
  publisher={Cambridge University Press}
}

@article{NIP2,
  title={On {NIP} and invariant measures},
  author={Hrushovski, Ehud and Pillay, Anand},
  journal={Journal of the European Mathematical Society},
  volume={13},
  number={4},
  pages={1005--1061},
  year={2011},
  publisher={European Mathematical Society Publishing House},
  doi={10.4171/JEMS/274},
}

@article{NIP1,
  title={Groups, measures, and the {NIP}},
  author={Hrushovski, Ehud and Peterzil, Ya'acov and Pillay, Anand},
  journal={Journal of the American Mathematical Society},
  doi={10.1090/S0894-0347-07-00558-9},
  volume={21},
  number={2},
  pages={563--596},
  year={2008}
}

@article{CP12,
title = {Connected components of definable groups and o-minimality I},
journal = {Advances in Mathematics},
volume = {231},
number = {2},
pages = {605-623},
year = {2012},
issn = {0001-8708},
doi = {10.1016/j.aim.2012.05.022},
author = {Conversano, Annalisa and Pillay, Anand},
}

@book{Guide,
  title={A guide to NIP theories},
  author={Simon, Pierre},
  year={2015},
  publisher={Cambridge University Press},
  doi={10.1017/CBO9781107415133},
}

@online{HR25,
  title={On idempotent measure conjecture and decomposition of invariant measures},
  author={Hoffmann, Daniel and Rzepecki, Tomasz},
  year={2025},
  eprinttype = {arxiv},
  eprint = {2511.22945},
}

@article{simon2015invariant,
  title={Invariant types in NIP theories},
  author={Simon, Pierre},
  journal={Journal of Mathematical Logic},
  volume={15},
  number={02},
  pages={1550006},
  doi={10.1142/S0219061315500063},
  year={2015},
  publisher={World Scientific}
}

@online{rzepecki2018bounded,
  title={Bounded invariant equivalence relations},
  author={Rzepecki, Tomasz},
  year={2018},
  eprinttype = {arxiv},
  eprint = {1810.05113},
  note = {PhD thesis}
}

@article{shelah2008minimal,
  title={Minimal bounded index subgroup for dependent theories},
  author={Shelah, Saharon},
  journal={Proceedings of the American Mathematical Society},
  volume={136},
  number={3},
  pages={1087--1091},
  year={2008},
  doi={10.1090/S0002-9939-07-08654-6},
}

@book{Glasner:Proximal_flows,
	title = {Proximal Flows},
	series = {Lecture Notes in Mathematics},
	number = {517},
	publisher = {Springer-Verlag},
	author = {Glasner, Shmuel},
	year = {1976},
    doi = {10.1007/BFb0080139},
}

@article{chernikov2014external,
  title={External definability and groups in NIP theories},
  author={Chernikov, Artem and Pillay, Anand and Simon, Pierre},
  journal={Journal of the London Mathematical Society},
  volume={90},
  number={1},
  pages={213--240},
  year={2014},
  doi={10.1112/jlms/jdu019},
  publisher={Oxford University Press}
}

@article{pillay2013topological,
  title={Topological dynamics and definable groups},
  author={Pillay, Anand},
  journal={The Journal of Symbolic Logic},
  volume={78},
  number={2},
  pages={657--666},
  year={2013},
  doi={10.2178/jsl.7802170},
  publisher={Cambridge University Press}
}

@article{hrushovski2012note,
  title={A note on generically stable measures and fsg groups},
  author={Hrushovski, Ehud and Pillay, Anand and Simon, Pierre},
  journal={Notre Dame Journal of Formal Logic},
  doi={10.1215/00294527-1814705},
  year={2012},
  pages={599--605},
}

@article{yao2023minimal,
  title={On minimal flows and definable amenability in some distal NIP theories},
  author={Yao, Ningyuan and Zhang, Zhentao},
  journal={Annals of Pure and Applied Logic},
  volume={174},
  number={7},
  pages={103274},
  year={2023},
  publisher={Elsevier},
  doi={j.apal.2023.103274},
}

@article{GJK23,
  title={Bohr compactifications of groups and rings},
  author={Gismatullin, Jakub and Jagiella, Grzegorz and Krupi{\'n}ski, Krzysztof},
  journal={The Journal of Symbolic Logic},
  volume={88},
  number={3},
  pages={1103--1137},
  year={2023},
  doi={10.1017/jsl.2022.10},
}

@article{pillay2016minimal,
  title={On minimal flows, definably amenable groups, and o-minimality},
  author={Pillay, Anand and Yao, Ningyuan},
  journal={Advances in Mathematics},
  volume={290},
  pages={483--502},
  year={2016},
  doi={10.1016/j.aim.2015.12.010},
  publisher={Elsevier}
}

\end{document}